\documentclass[11pt,a4paper]{amsart}
\usepackage{amssymb}
\usepackage{amsthm}
\usepackage{amsmath}
\usepackage{amsxtra}
\usepackage{latexsym}
\usepackage{mathrsfs}
\usepackage[all,cmtip]{xy}
\usepackage[all]{xy}
\usepackage{enumitem}
\usepackage{xcolor}
\usepackage{graphicx}
\usepackage{comment}
\usepackage{mathabx,epsfig}
\usepackage{stmaryrd}
\usepackage{comment}
\usepackage{pst-node}

\usepackage{tikz-cd} 

\usepackage{hyperref}

\newcommand{\RNum}[1]{\uppercase\expandafter{\romannumeral #1\relax}}

\newtheorem{thm}{Theorem}[section]
\newtheorem{lem}[thm]{Lemma}

\newtheorem{quest}[thm]{Question}

\newtheorem{prop}[thm]{Proposition}
\newtheorem{cor}[thm]{Corollary}

\theoremstyle{definition}
\newtheorem{defn}[thm]{Definition}
\newtheorem{rmk}[thm]{Remark}

\numberwithin{equation}{section}

\newcommand\be{\begin{equation}}
	\newcommand\ba{\begin{eqnarray}}
		\newcommand\ee{\end{equation}}
	\newcommand\ea{\end{eqnarray}}

\def\C{{\mathbb C}}
\def\Q{{\mathbb Q}}
\def\R{{\mathbb R}}
\def\Z{{\mathbb Z}}
\def\P{{\mathbb P}}
\def\A{{\mathbb A}}
\def\N{{\mathbb N}}

\DeclareMathOperator{\Prep}{Prep}

\DeclareMathOperator{\Per}{Per}

\DeclareMathOperator{\Gal}{Gal}
\DeclareMathOperator{\Orb}{Orb}

	{\end{list}}

\title[Common zeros of iterated morphisms]
{Towards common zeros of iterated morphisms}
\thanks{The second author was supported in part by NSERC grant RGPIN-2022-02951.}

\author[Chatchai Noytaptim]{Chatchai Noytaptim}
\address{University of Waterloo \\
	Department of Pure Mathematics \\
	Waterloo, Ontario \\
	Canada  N2L 3G1}
\email{cnoytaptim@uwaterloo.ca, chatchai.noytaptim@gmail.com}
\author{Xiao Zhong}
\address{University of Waterloo \\
	Department of Pure Mathematics \\
	Waterloo, Ontario \\
	Canada  N2L 3G1}
\email{x48zhong@uwaterloo.ca}
\date{\today}
\subjclass[2020]{37P05, 37P30, 37P50}
\keywords{arithmetic dynamics,  arithmetic equidistribution, compositional independence, free semigroup, polynomial decomposition}
\begin{document}
	\maketitle
	\begin{abstract} 
Recently, the authors have proved the finiteness of common zeros of two iterated rational maps under some compositional independence assumptions. In this article, we 
advance towards a question of Hsia and Tucker \cite[Question 19]{HT17} on a Zariski non-density of common zeros of iterated morphisms on a variety. More precisely, we provide an affirmative answer in the case of {H}\'{e}non type maps on $\mathbb{A}^2$, endomorphisms on $(\mathbb{P}^1)^n$, and polynomial skew products  on $\mathbb{A}^2$ defined over $\overline{\mathbb{Q}}$. As a by-product, we prove a Tits' alternative analogy for semigroups generated by two regular polynomial skew products. 
	\end{abstract}
\section{Introduction}
Given a variety $V$ defined over $\C$. Two dominant morphisms $f:V\rightarrow V$ and $g:V\rightarrow V$ are said to be compositionally independent if there are positive integers $k$ and $l$ such that
$$\phi_1\circ \phi_2\circ...\circ\phi_k=\phi_1\circ\phi_2\circ...\circ \phi_l$$ where $\phi_i\in \{f,g\}$, then $k=l$. In other words, the semigroup generated by $f$ and $g$ under the composition is isomorphic to the free semigroup of two generators. If $f$ and $g$ are not compositionally independent, we call them compositionally dependent.\\
\indent Here and what follows, we denote by $f^n:=f\circ f\circ...\circ f$ the $n$-fold composition of  $f$ with itself. In \cite{HT17}, Hsia and Tucker posed the following question:
\begin{quest} \label{quest:HTmainquestion} \cite[Question 19]{HT17} Let $V$ be a variety defined over $\C$ and let $f,g : V\rightarrow V$ be two dominant compositionally independent morphisms. Let $c : V\rightarrow V$ be any morphism. Is it true that the set of $\lambda\in V(\C)$ such that 
$$f^n(\lambda)=g^n(\lambda)=c(\lambda)$$ must be contained in a proper Zariski closed subset of $V$?
\end{quest}
To avoid an ambiguity, let us denote $f^{\langle n \rangle}$ the  $n$th power of polynomials $f$ in this paragraph. Question \ref{quest:HTmainquestion} has a long history dating back to the work of Bugeaud-Corvaja-Zannier \cite{BCZ03} providing an upper bound for the greatest common divisor of $a^{\langle n \rangle}-1$ and $b^{\langle n \rangle}-1$ where $a$ and $b$ are multiplicatively independent integers $\geq2$. Their work serves as an application of the Schmidt subspace theorem, a beautiful result in Diophantine approximation. This framework was later extended by Ailon-Rudnick \cite{AR04} and Ostafe \cite{Os16} to the function field setting with integers replaced by polynomials. In fact, they proved that if $f,g\in \C[x]$ are two multiplicative independent polynomials, then there exists $h\in \C[x]$ such that $\mathrm{gcd}(f^{\langle m \rangle}-1,g^{\langle n \rangle}-1)|h$  for all $m,n\geq1$. At the heart of the proof, they employ a well-known conjecture of Lang (which proved by Ihara-Serre-Tate \cite{La65, La83}) on the finiteness of the intersection of an irreducible non-special curve\footnote{Here, a special curve means it is a translation by a torsion point of an algebraic subgroup of $\C^*\times \C^*$. That is, it has the form $x^my^n-\zeta=0$ or $x^m-\zeta y^n=0$ where $\zeta\in\mu_{\infty}$.} in $\C^*\times\C^*$ with the roots of unity $\mu_{\infty}\times\mu_{\infty}$ applying to the rational curve $\{(f(x),g(x)) :x\in\C\}$. The boundedness of $\deg\mathrm{gcd}(f^{\langle m \rangle}-1,g^{\langle n \rangle}-1)$ is due to the work of Beukers and Smyth \cite{BS02}.  Motivated by the aforementioned works, Hsia and Tucker \cite{HT17} were the first to use an entirely different approach and tools to study the dynamical analogue of the greatest common divisors of iterated polynomial maps. Roughly speaking, their method hinges on an arithmetic equidistribution of small points, Northcott's theorem, and deep results from Diophantine geometry. \\
\indent Hsia-Tucker \cite{HT17} originally provided the  answer to Question \ref{quest:HTmainquestion} in affirmative  when $V=\C$ and $f,g,c$ are polynomials in $\C[x]$.  Recently, Noytaptim-Zhong \cite{NZ24} have verified the same question when $V=\P^1_{\C}$ and $f,g,c$ are rational functions in $\C(x)$. Thus, Question \ref{quest:HTmainquestion} has been answered in the affirmative in the following cases:
\begin{table}[!ht]  
    \centering
    \begin{tabular}{|l|l|l|l|}
\cline{2-4}
    \multicolumn{1}{c|}{} &$\deg f>1$&$\deg f>1$& $\deg f=1$  \\
        \multicolumn{1}{c|}{}  &$\deg g>1$&$\deg g=1$& $\deg g=1$\\
         \hline
         $V=\C$&\cite{HT17}&\cite{HT17}&\cite{HT17}\\ $f,g,c\in \C[x]$&Proposition 8&Proposition 9&Theorem 2\\ \text{polynomials}& & &\\
         \hline
         $V=\P^1_{\C}$&\cite{NZ24}&\cite{NZ24}&\cite{NZ24}\\ $f,g,c\in \C(x)$&Proposition 4.3&Proposition 4.4&Theorem 1.3\\ \text{rational functions}& & &\\
         \hline
    \end{tabular}
\end{table}

\indent As one might expect, Question \ref{quest:HTmainquestion} remains far from being fully resolved at this time due to the absence of a reasonable criterion for a semigroup of dominant morphisms to be free. In this article, we make a partial progress towards the Question \ref{quest:HTmainquestion} in certain cases, namely polynomial automorphims on $\A^2_{\overline{\Q}}$,  endomorphisms on $(\P^1_{\overline{\Q}})^n$, and polynomial skew products on $\A^2_{\overline{\Q}}$.\\
\indent The dynamics of polynomial automorphisms of $\A^2_{\C}$ has been explored extensively in the past decades, e.g., \cite{BS91}, \cite{FS99}, \cite{Hu86}, and references therein.  Notably, Friedland and Milnor \cite{FM89} have classified three different types of polynomial automorphisms on $\A^2_{\C}$. More precisely, any polynomial automorphism is conjugate to either one of the following forms : \begin{enumerate}
\item[(a)] an affine map of the form $a(x,y)=(a_1x+b_1y+c_1, a_2x+b_2y+c_2)$ for some constants $a_i,b_i,c_i$ such that $a_1b_2-a_2b_1\neq0$;
\item[(b)] an elementary automorphism of the form $e(x,y)=(ax+P(y),by+c)$ for some constants $a,b,$ and $c$ with $ab\neq0$ and for some polynomial function $P(y)$;
\item[(c)] a finite composition of regular polynomial automorphism of degree $\geq 2$ of the form 
$$h_j(x,y)=(y,P_j(y)-\delta_j x)$$ where $P(y)$ is a polynomial of degree $\geq 2$.
\end{enumerate}
Here the degree of  polynomial automorphisms denotes the maximum of the degree of its two components.
Any polynomial automorphism on $\A^2_{\C}$  of type (c) above  will be called {H}\'{e}non type and its degree is defined to be $(\deg P_1)(\deg P_2)\cdots (\deg P_m)$ for some $m\geq1$. Thus an  {H}\'{e}non type map is of the form $$h(x,y)=(y,P(y)-\delta x)$$ where $\deg P=d\geq 2$. Note that the rational map $\tilde{h} : \P^2 \dashrightarrow \P^2$ induced by  $h$ is $$\tilde{h}[x:y:z]=[yz^{d-1}:z^dP\left(\frac{y}{z}\right)-\delta xz^{d-1}:z^d]$$ which has exactly one point of indeterminacy $p_+=[1:0:0]$. Furthermore, the hyperplane $\{z=0\}$ at infinity is mapped to the point $p_{-}=[0:1:0]$ which is not in the indeterminacy locus of $\tilde{h}$. Similarly, the rational map $\tilde{h}^{-1} : \P^2\dashrightarrow\P^2$ induced by the inverse $h^{-1}(x,y)=(-\delta^{-1}y+\delta^{-1}P(x),x)$ has one point of indeterminacy $p_{-}$ and it maps the hyperplane $\{z=0\}$ at infinity to the  point $p_{+}$. The loci of indeterminacy of $\tilde{h}
$ and $\tilde{h}^{-1}$ are disjoint and {H}\'{e}non type $h$ is also called regular, see \cite{Sib99}. \\
\indent Our first main result is an evidence supporting Questions \ref{quest:HTmainquestion} in the case of {H}\'{e}non type map on $\A^2_{\overline{\Q}}$.
\begin{thm}\label{thm:Henontypemain}  Let $f$ and $g$ be two polynomial automorphisms of {H}\'{e}non type  of the affine plane $\A^2$ defined over a number field $K$. Let $c : \A^2 \rightarrow \A^2$ be a morphism defined over $K$ which is not a compositional power of $f$ and $g$. Suppose that $f$ and $g$ are compositionally independent. Then the set of $p\in \A^2_{\overline{K}}$ satisfying 
$$f^m(p)=g^n(p)=c(p)$$ for some positive integers $m$ and $n$ is not a Zariski dense subset of $\A^2$.
\end{thm}

 Let $f_1,f_2,...,f_n$ be rational functions of degree $d\geq 2$. Then we consider  $F=(f_1,f_2,...,f_n)$  an  endomorphism on $(\P^1)^n$ which given by their coordinatewise action $$F(x_1,x_2,...,x_n)=(f_1(x_1),f_2(x_2),...,f_n(x_n)).$$ As endomorphisms on $(\P^1_{\overline{\mathbb{Q}}})^n$ are split-type, it often confines the complexity of dynamics to dimension one which many tools are available. The dynamics of endomorphisms on $(\P^1)^n$ has been extensively studied by many authors, e.g., \cite{GNY18}, \cite{GNY19}, \cite{Mi13}, \cite{MSW23}, \cite{Zh23}. 
Then we have the following result which claims the non-Zariski density of common zeros of two iterated endomorphisms on $(\P^1_{\overline{\mathbb{Q}}})^n$. For $n=1$, it was previously known in \cite[Theorem 1.4]{NZ24}.

\begin{thm}\label{thm:commonzerosofendomorphism} Given  endomorphisms $F=(f_1,...,f_n)$ and $ G=(g_1,...,g_n),$  on $(\P^1)^n$ defined over a number field $K$ of degree $d\geq 2$. Suppose that the semigroup generated by $F$ and $G$ is free. Let $R=(r_1,r_2,...,r_n)$ be any endomorphism on $(\P^1)^n$ defined over $K$ such that $R$ is not compositionally power of $F$ or $G$. Let  $X:=(x_1,x_2,...,x_n)\in (\P^1_{\overline{K}})^n.$ Then there exist positive integers $k$ and $l$ with the property that the solutions to \begin{equation}F^k(X)=G^l(X)=R(X)\notag\label{eq:mainthmeq}\end{equation} form a Zariski-non-dense set in $(\P^1_{\overline{K}})^n$.
\end{thm}
Following Jonsson \cite[\S1]{Jo99}, a polynomial skew product on $\C^2$ of degree $d\geq 2$ is a map of the form $f(x,y)=(p(x),q(x,y))$, with respect to some choice of coordinates $(x,y)$, where $p$ and $q$ are polynomials of degree $d$ where $p(x)=x^d+O(x^{d-1})$ and $q(x,y)=y^d+O_x(y^{d-1}).$ In fact, $f$ is a regular polynomial map of $\C^2$. In other words, it extends to a holomorphic mapping of $\P^2$ which we still denote by $f$. Consider homogeneous coordinates $[x:y:z]$ on the projective plane  $\P^2_{\C}$, we identify $(x,y)\in \mathbb{A}^2_{\C}$ with $[x:y:1]\in \P^2_{\C}$. The extension of $f$ is thus given by $$f([x:y:z])=\left[z^dp\left(\frac{x}{z}\right):z^dq\left(\frac{x}{z},\frac{y}{z}\right):z^d\right].$$ 
\indent 
The dynamics of polynomial skew products have garnered significant attention recently, e.g., \cite{AB23}, \cite{ABDPR16}, \cite{JZ23}, \cite{JZ20}, \cite{UK20}. On one hand, these maps preserve a foliation of vertical lines, allowing the use of one-dimensional tools. On the other hand, their dynamics differ markedly from those in one dimension. For instance, polynomial skew products have been used to construct the first example of endomorphisms on $\P^k$, $k \geq 2$, that possess wandering Fatou domains in \cite{ABDPR16}, a phenomenon not observed in one-dimensional dynamics due to Sullivan’s non-wandering domain theorem. Consequently, polynomial skew products are considered a foundational gateway for studying higher-dimensional dynamics.

In this paper, we construct canonical height associated to polynomial skew products and further study some of its arithmetic properties, see sections \S \ref{sec:TitAlternativepolyskew}, \S \ref{sec:arithmeticdynpolyskewprods}, and Appendix \ref{appenA} for our main finding. This establishment leads to the following theorem which confirms that the question \ref{quest:HTmainquestion} also holds true for polynomial skew products defined $\overline{\mathbb{Q}}$.

\begin{defn}
    We call a polynomial endomorphism $F(x,y)$ on $\A^2$ a regular polynomial skew product of degree $d> 1$ defined over a field $K$ if there exists an affine linear automorphism $\sigma$ defined over $K$ such that 
    $$ \sigma \circ F\circ \sigma^{-1} (x,y) = (f(x), g(x,y))$$
    where $f(x)$ and $g(x,y)$ are polynomials defined over $K$, $\deg(f) = \deg(g) = d$ and $g(x,y)$ contains a non-trivial $y^d$ term.
\end{defn}
\begin{rmk}
    This guarantees that $F(x,y)$ can be extended to an endomorphism on $\P^2$.
\end{rmk}

\begin{thm} \label{thm:commonzerosofpolyskewprods}Let $F(x,y)$ and $G(x,y)$ be regular polynomial skew products of degree $d\geq 2$ on the affine plane $\A^2$ defined over a number field $K$. Suppose that $F$ and $G$ are compositionally independent and $C : \A^2\rightarrow \A^2$ (defined over $K$) is a morphism that is not a compositionally power of $F$ or $G$. Then there are positive integers $m$ and $n$ with the property that the set $\{(x,y)\}\subset \A^2_{\overline{K}}$ satisfying $$F^m(x,y)=G^n(x,y)=C(x,y)$$ is not a Zariski dense subset of $\A^2$.
\end{thm}

The proof of Theorem\ref{thm:Henontypemain}, Theorem \ref{thm:commonzerosofendomorphism}, and Theorem \ref{thm:commonzerosofpolyskewprods} follows along the same lines. Let us discuss the case of {H}\'{e}non maps. The outline of the proof can be divided into four simple steps. Proceeding by a contradiction, suppose that there exists a generic sequence $\{p_i\}\subset \A^2_{\overline{K}}$ satisfying $f^{m_i}(p_i)=g^{n_i}(p_i)=c(p_i)$. Then
\begin{enumerate}
    \item[ \textbf{(I)}] the sequence  $\{p_i\}_{i\geq 1}$ is dynamically small with respect to both $f$ and $g$. This is done in Proposition \ref{prop:smallpoints}.
 \item[ \textbf{(II)}] applying the arithmetic equidistribution Theorem \ref{thm:equidistributionofsmallpoints} yields the equality of measures $\mu_{f,v}=\mu_{g,v}$ at all places $v\in M_K$.
  \item[ \textbf{(III)}] {H}\'{e}non type maps $f$ and $g$ must share the same set of periodic points.
   \item[ \textbf{(IV)}]  $f$ and $g$ must be compositionally dependent.
\end{enumerate}
The compositional dependence of $f$ and $g$ contradicts the assumption. \\

The first three steps are motivated by the seminal work of Baker-DeMarco \cite{BD11}, Petsche-Szpiro-Tucker \cite{PST12}, and Yuan-Zhang \cite{YZ17}. This strategy is commonly used in dynamical unlikely intersections questions. However, the last step is a notoriously hard problem in algebraic dynamics in general. In the context of {H}\'{e}non maps, Dujardin-Favre \cite{DF17} deduced the compositional dependence of two {H}\'{e}non maps (sharing the same set of periodic points) from an earlier work of Lamy \cite{Lam01} on the Tits alternative of $\mathrm{Aut}[\mathbb{C}^2]$. Inspired by the one-dimensional result of \cite[Corollary 4.12]{BHPT}, compositional dependence of two endomorphisms on $(\mathbb{P}^1)^n$ (when their set of preperiodic points coincide) can be obtained via a careful induction argument. The most intriguing situation is for regular polynomial skew products, by applying a detailed analysis with Ritt's polynomial decomposition Theorem and specialization arguments, we manage to deduce an analogy Tits alternative of regular polynomial skew products. 
\begin{thm}[See Corollary \ref{cor: general-skew-product-case} and Remark \ref{rmk: Tits-alternative-skew}]
    Let $F_1$ and $F_2$ be two regular polynomial skew products defined over $\C$. Then two following possibilities occur: either
    \begin{enumerate}
        \item  $\Prep(F_1) = \Prep(F_2)$; or
        \item  the semigroup $\langle F_1, F_2 \rangle$ generated by compositions contains a non-abelain free subsemigroup.
    \end{enumerate}
\end{thm}

\noindent This solves the last step of the polynomial skew products case.
\section*{ Acknowledgements}
We are grateful to Liang-Chung Hsia and Tom Tucker who posed Question \ref{quest:HTmainquestion}  in \cite{HT17}.
Their work greatly motivates us to pursue this research direction. We to thank Jason Bell for helpful discussion and comments on an earlier version of the draft. The second author was supported in part by NSERC grant RGPIN-2022-02951.
	\section{Arithmetic dynamics of {H}\'{e}non type maps}
 \subsection{Absolute values on number fields}\label{subsec:absolutevalues}
Let $K$ be a number field and $\overline{K}$ be a fixed algebraic closure of $K$. Let $M_K$ be the set of places of $K$, that is, an equivalence class of nontrivial absolute values on $K$. For each place $v\in M_K$, $K_v$ denotes the corresponding completion of $K$ with respect to the absolute value $|\cdot|_v$ (determined up to equivalence). The absolute value $|\cdot|_v$ on $K$ is either standard archimedean absolute value or the $p$-adic absolute value satisfying $|p|_p=p^{-1}$ when restricted to $\Q$. For any non-zero element $x\in K$, the product formula 
$$\prod_{v\in M_K}|x|^{[K_v:\Q_v]}_v=1$$ holds. Denote by $\C_v$ the completion of the algebraic closure $\overline{K}$ of the number field $K$ with respect to the absolute value $|\cdot|_v$. By abuse notation, we still denote by $|\cdot|_v$ the unique extension to $\C_v$ of the absolute value on $K$. Note that $\C_v$ is both complete and algebraically closed.
\subsection{Green's functions}  Let $f$ and $g$ be two polynomial automorphisms of {H}\'{e}non type of the affine plane $\A^2$ defined over a number field $K$ of degree $d\geq 2$. 
We follow a construction and discussion of Green's functions associated to polynomial automorphisms on $\A^2_{K}$ of Ingram, Hsia and Kawaguchi \cite{In14},  \cite{Ka06}, \cite{Ka13}, and \cite{HK18}.  Set  $\|(x,y)\|_v=\max\{|x|_v,|y|_v\}$ and $\log^{+}(x)=\max\{\log(x),0\}$. As in \cite{Br65} and \cite{Hu86} (in archimedean setting),  the $v$-adic Green functions $G^+_{f,v}, G^-_{f,v}:\C^2_v\rightarrow\mathbb{R}_+$ are continuous functions and defined by
\begin{equation}
G_{f,v}^+(x,y)=\lim_{n\rightarrow\infty}\frac{1}{d^n}\log^+\|f^n(x,y)\|_v\quad\label{eq:Greenplus}\end{equation} and \begin{equation}\quad G_{f,v}^-(x,y)=\lim_{n\rightarrow\infty}\frac{1}{d^n}\log^+\|f^{-n}(x,y)\|_v\label{eq:Greenminus}\end{equation}
where $d$ is the common degree of $f$ and $f^{-1}.$ It was shown by Kawaguchi \cite[Theorem A(1)]{Ka13} that the limits $G^+_{f,v}$ and $G^-_{f,v}$ exist and nonnegative. They also satisfy the invariant properties 
\begin{equation}
    G_{f,v}^+\circ f=dG_{f,v}^+\quad\text{and}\quad G_{f,v}^-\circ f^{-1}=dG_{f,v}^-.\label{eq:invariantgreen}\end{equation}
We record interesting properties of Green functions for future reference. For archimedean $v=\infty$, we write $G^+_f$ and $G^{-}_f$ in place of $G^{+}_{f,\infty}$ and $G^{-}_{f,\infty}$, respectively, for simplicity.
\begin{thm} \label{thm:semipositivegreenfns} Let $G_{f,v}:=\max\{G^+_{f,v},G^-_{f,v}\}$ be the continuous and non-negative function on $\A^2_{K}$.
 The following statements hold true:
 \begin{enumerate}
     \item  $G_{f,v}-\log^+\|\cdot\|_v$ extends to a continuous function to $\P^2_K$.
     \item  $G_{f,v}=\log^+\|\cdot\|_v$ for all but finitely many $v\in M_K$.
     \item  $G_f$ is plurisubharmonic function for each archimedean place $v$.
     \item  $G_{f,v}$ is the uniform limit of the sequence of continuous function $$\max\{d^{-n}\log^+\|f^n\|_v,d^{-n}\log^+\|f^{-n}\|_v\}.$$
 \end{enumerate}
\end{thm}
\begin{proof}
For (1), it was discussed in \cite{BS91} for archimedean and in \cite{In14} and \cite{Ka13} for non-archiemedean. For (2), polynomial automorphisms on $\A^2_K$ has good reduction at $v$ except for finitely many places $v$, see \cite[Theorem B(1)]{Ka13}.  For (3), Bedford and Smillie proved in \cite[Proposition 3.8]{BS91}. For (4), it was verified by \cite[Proposition 1.10(3)]{DF17}.
\end{proof}
\subsection{Canonical height functions} Recall the standard Weil height $h:\A^2_{\overline{K}}\rightarrow\R$ of any point $p\in \A^2_{\overline{K}}$ is defined by
$$h(p)=\frac{1}{\deg(p)}\sum_{v\in M_K}\sum_{q\in \mathrm{Gal}(\overline{K}/K)\cdot p}N_v\log^+\|q\|_v$$ where $N_v=[K_v:\Q_v]$ denotes the degree of the field extension of the completion of $K$ over the completion of $\Q$ relative to $|\cdot|_v$ and  $\deg(p)$ denotes the cardinality of $\mathrm{Gal}(\overline{K}/K)\cdot p$ the Galois orbit of $p$.\\

\noindent Then we also define the canonical height $\hat{h}_f:\A^2_{\overline{K}}\rightarrow\R$ associated to $f$ is thus given by
$$\hat{h}_f(p)=\frac{1}{\deg(p)}\sum_{v\in M_K}\sum_{q\in \mathrm{Gal}(\overline{K}/K)\cdot p}N_v (G_{f,v}^+(q)+G_{f,v}^-(q)).$$  However, it is more convenient for us to work with a compatible height (cf. Theorem \ref{thm:canonicalheight}(1)) $$\tilde{h}_f(p)=\frac{1}{\deg(p)}\sum_{v\in M_K}\sum_{q\in \mathrm{Gal}(\overline{K}/K)\cdot p}N_vG_{f,v}(q)$$ where $G_{f,v}(\cdot)=\max\{G_{f,v}^+(\cdot),G_{f,v}^-(\cdot)\}$. Here and what follows, we refer $\tilde{h}$ as our canonical height associated to $f$.
\begin{rmk}
We shall soon see that the canonical height $\tilde{h}_f$  is associated to a continuous semipositive adelic metric of the ample line bundle $\mathcal{O}_{\P^2}(1)$ (in the sense of Zhang \cite{Zh95}).
\end{rmk}
\begin{thm} \label{thm:canonicalheight}  Let $K$ be a number field and $h$ be the Weil height on $\A^2_{\overline{K}}$. Let $\hat{h}_f,\tilde{h}_f$ be as  above. Then
	\begin{enumerate}
		\item $\hat{h}_f/2\leq \tilde{h}_f\leq \hat{h}_f$.
		\item $\tilde{h}_f=h+O(1)$
		\item For $p\in \A^2_{\overline{K}}$, $\tilde{h}_f(p)=0$ if and only if $p\in \mathrm{Per}(f)$.\\ Here $\mathrm{Per}(f)$ denotes the set of all periodic points of $f$.
	\end{enumerate}
\end{thm}
\begin{proof}  For (1), it is an easy observation that $(a+b)/2\leq \max\{a,b\}\leq a+b$ for any real numbers $a$ and $b$. For (2) and (3), see \cite[Theorem 6.3 and Proposition 6.5]{Ka13}.
\end{proof}
Given $c$ is any morphism $\A^2\rightarrow\A^2$ defined on $K$. Our first result asserts that a generic sequence $\{p_n\}\subset \A^2_{\overline{K}}$ satisfying the equation
 $$f^n(p_n)=g^n(p_n)=c(p_n)$$
forms a sequence of small points. 
\begin{prop} \label{prop:smallpoints}  Let $K$ be a number field. Suppose that $\{p_n\}$ is a generic sequence in $\A^2_{\overline{K}}$ satisfying $$f^n(p_n)=c(p_n)$$ for all $n\geq 1$ where $f$  is any polynomial automorphism defined over $K$ and $c$ is an arbitrary self-morphism on $\A^2$ defined over $K$. Then the sequence $\{p_n\}_{n\geq1}$ is $\tilde{h}_f$-small. More precisely,  $$\lim_{n\rightarrow\infty}\tilde{h}_f(p_n)=0.$$
\end{prop}
\begin{proof}  Using the invariant property (\ref{eq:invariantgreen}) of the Green functions $G_{f,v}^+$ and $G_{f,v}^-$, it follows that  \begin{align}\tilde{h}_f(f^n(p_n))&=\frac{d^n}{\deg(p_n)}\sum_{v\in M_K}N_v\max\{G_{f,v}^+(p_n),G_{f,v}^-(p_n)\}\notag\\&=d^n\tilde{h}_f(p_n).\label{eq:canheightiterate}\end{align} 
For points $p=(x,y)\in\A^2_{\overline{K}}$, fix an embedding $\A^2_{\overline{K}}\hookrightarrow\P^2_{\overline{K}}$ by sending $(x,y)$ to $[x:y:1]$. 
 We recall  a property of rational maps $\phi : \P^2\rightarrow\P^2$ which states
\begin{equation}h(\phi(p))\leq (\deg \phi)h(p)+O(1)\notag\end{equation} for each $p\in \P^2_{\overline{\mathbb Q}}$ away from the indeterminacy locus of $\phi$ (cf. \cite[\S B.2]{HS00}).   Applying Theorem \ref{thm:canonicalheight} (2), we have 
\begin{equation}\tilde{h}_f(c(p_n))\leq  (\deg c)\tilde{h}_f(p_n)+ O(1)\label{eq:heightbound}\end{equation}
Combining equation (\ref{eq:canheightiterate}) and inequality (\ref{eq:heightbound}), we obtain that
$$(d^n-(\deg c))\tilde{h}_f(p_n)\leq O(1)$$ where the implied constant is independent of $n$. As $d\geq 2$ and letting $n\rightarrow\infty$, it follows that $\tilde{h}_f(p_n)$ must tend to zero.
\end{proof}
\subsection{Arithmetic Equidistribution} In order to apply Yuan's equidistribution  theorem, it is necessary to associated height to semipositive adelic metrics on ample line bundles in reminiscent of \cite{CL11} and \cite{Zh95}.  
For our purpose, we slightly modify the definition of semipositive adelic metrics on line bundles, for usual convention the reader may consult, for example \cite[\S 4]{GHT15}.
The following convention is due to Dujardin-Favre \cite{DF17}. Here, we fix an embedding $\A^2_{\overline{K}}\hookrightarrow \P^2_{\overline{K}}.$
\begin{defn}\label{defn:semipositadelicmetricsdef}
A collection $\{G_v\}$ of functions $G_v:\A^2_{\overline{K}}\rightarrow\R$ is called a semipositive adelic metric on the ample line bundle $\mathcal{O}_{\P^2}(1)$ is the has the following properties:
\begin{enumerate}
	\item[(SP 1)]  the function $G_v(p)-\log^+\|p\|_v$ extends to a continuous function in $\P^2_{K}$ for each place $v\in M_K$;
	\item[(SP 2)] $G_v(p)=\log^+\|p\|_v$ for all but finitely many $v$; 
	\item[(SP 3)] $G_v$ is plurisubhamonic for each archimedean $v$;
	\item[(SP 4)] for each non-archimedean $v$, the function $G_v$ is a uniform limit of positive multiple functions of the form $\log\displaystyle\max_{1\leq i\leq r}\{|P_i|_v\}$ with $P_i\in K[x_1,x_2]$.
\end{enumerate}
\end{defn}
We now can define a height function $h_G :\A^2_{\overline{K}}\rightarrow\R$ associated to the collection of semipositive adelic metric $\{G_{v}\}$ by
$$h_G(p):=\frac{1}{\deg(p)}\sum_{v\in M_K}\sum_{q\in \mathrm{Gal}(\overline{K}/K)\cdot p}N_vG_{v}(q)$$ such that $\sup_{p\in\A^2_{\overline{K}}}|h_G(p)-h(p)|<+\infty$.
\begin{prop} Let $f$ be any regular polynomial automorphism of degree $\geq 2$. The collection $\{G_{f,v}\}$ defines a semipositive adelic metric on the ample line bundle $\mathcal{O}_{\P^2}(1)$.
\end{prop}
\begin{proof} This follows from Theorem \ref{thm:semipositivegreenfns}.
\end{proof}
It is useful to recall related definitions for stating an arithmetic equidistribution of small points. We follow \cite{Yu08}.
\begin{defn} \begin{enumerate}
    \item[1.] A sequence $\{p_n\}_{n\geq 1}$ of algebraic points in $\P^2_{\overline{K}}$ is called generic if any inifnite subsequence is Zariski dense in $\P^2$. Equivalently, there is no infinite subsequence of $\{p_n\}$ is contained in a proper closed subvariety of $\P^2$.
    \item[2.] A sequence $\{p_n\}_{n\geq 1}$ of algebraic points in $\P^2_{\overline{K}}$ is said to be small with respect to the canonical height $\tilde{h}$ if $\tilde{h}_f(p_n)\rightarrow0$ as $n\rightarrow\infty$.
\end{enumerate}
\end{defn}
\begin{thm} \label{thm:equidistributionofsmallpoints} (Equidistribution for points of small height of {H}\'{e}non maps) Let $f$ be an automorphism of {H}\'{e}non type defined over a number field $K$ such that $\deg f\geq 2$. Suppose that $\{p_n\}$ is a generic sequence on $\P^2_{\overline{K}}$ such that $\tilde{h}_f(p_n)\rightarrow0$ with respect to the arithmetic canonical height $\tilde{h}_f$. Then, for any place $v\in M_K$, the sequence of probability measure supported equally on the Galois orbit of $p_n$
	$$\frac{1}{\deg(p_n)}\sum_{q\in \mathrm{Gal}(\overline{K}/K)\cdot p_n}\delta_q$$ converges weakly to $\mu_{f,v}$ on the Berkovich analytic space $\P^2_{\mathrm{Berk},\C_v}$. 
\end{thm}
\begin{proof} We directly apply the  equidistribution theorem of Yuan \cite[Theorem 3.1]{Yu08}. It suffices to check that $\tilde{h}_f(\P^2)=0.$ Recall Zhang's successive minima inequality \cite[Theorem 1.10]{Zh95} reads
	\begin{equation} \tilde{h}_{f}(\P^2)\leq \sup_{\mathrm{codim} Z=1}\inf_{p\in (\P^2\backslash Z)(\overline{K})}\tilde{h}_f(p)\label{eq:Zhangminima}\end{equation} where the supremum is taken over all closed subvarieties of codimension one in $\P^2$.
	The existence of the generic sequence $\{p_n\}$ with $\tilde{h}_f(p_n)\rightarrow0$, this immediately yields the  right hand side of the  inequality (\ref{eq:Zhangminima}) tends to $0$. Thus $\tilde{h}_f(\P^2)=0$ as desired. 
\end{proof}
\begin{rmk}  Let us record important notes in order. 
\begin{enumerate} \item[(i)] It was pointed out by Dujardin-Favre \cite[Corollary 1.19]{DF17} and Lee \cite[Corollary 7.5]{Le13} that the probability measure $\mu_{f,v}$ is $f$-invariant and has support on the so-called $v$-adic filled Julia set $K_{f,v}$ where 
	$$K_{f,v}=\{p\in \A^{2,\mathrm{an}}_{\C_v}: \sup_{n\in\Z}\|f^n(p)\|_v=O(1)\}.$$
Over $\mathbb{C}$, the probability measure satisfies $\mu_f:=dd^c G_f\wedge dd^c G_f$ and $f^*\mu_f=\mu_f$. Note also that $\mu_f:=dd^c G_f^+\wedge dd^cG^-_f$ is well-defined because of the continuity of the local potentials of $dd^cG_f^{\pm}$. For detailed exposition, the reader may infer \cite[\S 7.3]{FS99}. On the other hand, a non-archimedean analogue of measure $\mu_f$ was constructed as a probability measure on the Berkovich analytic space $\A^{2,\mathrm{an}}_{\C_v}$ in the sense of Chambert-Loir \cite{CL06, CL11}. In fact, the Green function $G_{f,v}$ induced a semipositive continuous metric $\|\cdot\|_{G_{f,v}}$ on line bundle $\mathcal{O}_{\P^2}(1)$ by setting $\|s\|_{G_{f,v}}:=\exp(-G_{f,v})$ where $s$ is the section of $\mathcal{O}_{\P^2}(1)$ corresponding to the constant function $1$ on $\A^{2,\text{an}}_{\mathbb{C}_v}$.
 \item[(ii)] Our Theorem \ref{thm:equidistributionofsmallpoints} is closely related to \cite[Theorem B]{Le13}. Strictly speaking, we can view $f$ and $f^{-1}$ as strongly regular pair in light of \cite[Definition 1.1]{Le13}. Hence, our Theorem is equivalent to Lee's equidistribution of small points in Theorem B. To make our article self-contained, we decide to state it as our current version.
 \end{enumerate}
\end{rmk} 
\subsection{Shared common zeros of iterated polynomial automorphisms} Recall the following result which was proved by Dujardin and Favre \cite[Lemma 6.3]{DF17}. It essentially says that the polynomial convex hull of $\mathrm{supp}(\mu_{f,v})$ is the filled Julia set $K_{f,v}$. Alternatively, we denote $K_{f,v}$ as the set $$\{p\in \A^{2,\mathrm{an}}_{\C_v}:\sup_{n\in\Z}|H^n(p)|_v<+\infty\}$$ where $|\cdot|_v$ is multiplicative seminorm on the ring $\C_v[X,Y]$ (corresponding to $p\in \A^{2,\mathrm{an}}_{\C_v}$) whose restriction on the ground field is the $v$-adic absolute value which we still denote by $|\cdot|_v$ by abuse notation. More details see \cite{Ber90}.
\begin{lem}\label{lem:convexhullofmeasure}  \cite[Lemma 6.3]{DF17} For any place $v$, $$\sup_{K_{f,v}}|P|_v=\sup_{\mathrm{supp}(\mu_{f,v})}|P|_v$$ for all $P\in \C_v[\A^2]$.
\end{lem}
We are now ready to state and proof the main result of this section. The following statement is more general than Theorem \ref{thm:Henontypemain} as a flexibility on degree of maps is allowed.
\begin{thm}  Let $f$ and $g$ be two polynomial automorphisms 
of the affine plane $\A^2$ defined over a number field $K$ such that at  least one of which has degree greater than one. Suppose that $f$ and $g$ are compositionally independent and a morphism $c:\A^2\rightarrow\A^2$ (defined over $K$) is not a compositional power of $f$ or $g$. Then the set of $p\in \A^2_{\overline{K}}$ such that $$f^m(p)=g^n(p)=c(p)$$ for some $m,n\in\N$ is contained in a proper Zariski closed subset of $\A^2$.\end{thm}
\begin{proof} [Proof sketch] \textbf{Case I :} Suppose that  $\deg f, \deg g>1$. In other words, $f$ and $g$ are both {H}\'{e}non type maps.\\
Suppose, on the contrary, that there is a generic sequence $\{p_i\}$ such that for every $i$, there exist $m_i$ and $n_i$ so that $f^{m_i}\neq c$ and $g^{n_i}\neq c$, and $f^{m_i}(p_i)=g^{n_i}(p_i)=c(p_i)$. As $i\rightarrow\infty$, $m_i$ and $n_i$ must both tend to $\infty$ because $f^{m_i}-c$ and $g^{n_i}-c$ have only finitely many zeros. It is clear from Proposition \ref{prop:smallpoints} that  $\{p_i\}$ is a sequence of small points with respect to both $f$ and $g$, more precisely,
	$$\lim_{i\rightarrow\infty}\tilde{h}_f(p_i)=\lim_{i\rightarrow\infty}\tilde{h}_g(p_i)=0.$$
We claim that $\mathrm{Per}(f)=\mathrm{Per}(g)$.
Applying the equidistribution theorem of small points for {H}\'{e}non maps (Theorem \ref{thm:equidistributionofsmallpoints}), we have that the sequence $\{p_i\}$ is equidistributed with respect to the probability measure $\mu_{f,v}$ and $\mu_{g,v}$. The uniqueness forces $\mu_{f,v}=\mu_{g,v}$ at all $v\in M_K$.  Using Lemma \ref{lem:convexhullofmeasure}, it follows that $K_{f,v}=K_{g,v}$. \\
	
	\noindent \textbf{Claim:} (arithmetic local-to-global property) Let $f$ be any regular polynomial automorphism defined over $K$. Then $p\in \A^2_{\overline{K}}$ is $f$-periodic  if and only if $p$ and all of its conjugates are contained in the $v$-adic filled Julia set $K_{f,v}$ for all $v\in M_K$. 
\begin{proof} [Proof of Claim] 
Let $L$ be a finite extension of $K$ and let $p$ be an element of $\A^2_L$. 
 Recall from \cite[Theorem 6.3(2)(5) and Proposition 6.5(2)] {Ka13} or Theorem \ref{thm:canonicalheight}(3) that $$\tilde{h}_f(p)=0\Longleftrightarrow p\in\mathrm{Per}(f).$$ Note also that $$\tilde{h}_f(p)=\frac{1}{\deg(p)}\sum_{\sigma : L\hookrightarrow\C_v}\sum_{v\in M_K}N_v\max\{G^+_{f,v}(p^{\sigma}),G^-_{f,v}(p^{\sigma})\}.$$ Since $G^+_{f,v}\geq0$ and $G^-_{f,v}\geq0$, it gives rise to the fact that 
$p$ is $f$-periodic if and only if $G_{f,v}(p^{\sigma})=\max\{G^+_{f,v}(p^{\sigma}),G^-_{f,v}(p^{\sigma})\}=0$ for all $v\in M_K$ and all $\sigma:L^2\hookrightarrow\C^2_v$.  By the limit definition (\ref{eq:Greenplus}) and (\ref{eq:Greenminus}) of $G^{\pm}_{f,v}$, it is equivalent to saying that $G_{f,v}(p)=0$ if and only if $\{f^{\pm n}(p)\}_{n\geq 0}$ is bounded with respect to $\|\cdot\|_v.$  Moreover, for each place $v\in M_K$, the $v$-adic filled Julia set $K_{f,v}$ captures points for which $G_{f,v}$ vanishes. This fact was verified by Kawaguchi \cite[Theorem 3.1]{Ka13} for non-archimedean $v$ and by Forn\ae ss-Sibony \cite[\S2]{Sib99} and \cite[Proposition 7.9]{FS99} for archimedean $v$. The  claim is established.
\end{proof}
	
	\noindent Take a $f$-periodic point $p$.  At any place $v$, it belongs to $K_{f,v}$ by local-to-global arithmetic property. So it  belongs to $K_{g,v}$. Recall that 
	$$\tilde{h}_g(p)=\frac{1}{\deg(p)}\sum_{v\in M_K}\sum_{q\in\mathrm{Gal}(\overline{K}/K)\cdot p}N_vG_{g,v}(q).$$ Since $K_{g,v}=\{G_{g,v}=0\}$, we then have $\tilde{h}_g(p)=0.$ This shows that $p$ is a $g$-periodic point. Hence we obtain $\mathrm{Per}(f)=\mathrm{Per}(g)$ as claimed.\\
\indent As explained in Dujardin-Favre \cite[Step 3]{DF17}, we know that $f$ and $g$ must share a common iterate and hence they are compositionally dependent. As a consequence, such small sequence $\{p_i\}$ must not be Zariski dense.\\

\noindent \textbf{Case II :} Without loss of generality, suppose that  $\deg f>1$ and $\deg g=1$. In other words, $f$ is {H}\'{e}non type map and $g$ is either affine map or an elementary automorphism of degree one.\\
\indent For any infinite  sequence of $\{p_i\}$ such that $f^{m_i}(p_i)=c(p_i)$ for all $i\in\N$, we have that $\{p_i\}$ is a sequence of small point with respect to $f$ by Proposition \ref{prop:smallpoints}. Moreover, any algebraic solution to $g^n(p)=c(p)$ is of bounded degree over $K$.  Hence Theorem \ref{thm:canonicalheight} (1) and the Northcott finiteness property \cite[Theorem A(i)]{Ka06}  implies that such a sequence $\{p_i\}$ is not generic (i.e., it is contained in a proper closed subvariety of $\A^2$).\end{proof}
\section{Arithmetic dynamics of endomorphisms on $(\P^1)^n$}
\subsection{Height functions associated to endomorphisms on $(\P^1)^n$} In this subsection, our absolute values on a number field $K$ are normalized as in \S \ref{subsec:absolutevalues}. Recall that the standard Weil height $h:\P^1_{\overline{K}}\rightarrow\R$  is given by
$$h(x)=\frac{1}{\deg(x)}\sum_{v\in M_K}\sum_{y\in\mathrm{Gal}(\overline{K}/K)\cdot x}N_v\log^+|y|_v$$ where $\deg(x)$ is again the cardinality of the orbit of $x$ under the action of the  absolute Galois group $\mathrm{Gal}(\overline{K}/K)$ of $K$. For each rational map $f(x)\in K(x)$ of degree $d\geq 2$ defined over $K$, the (global) canonical height function $\hat{h}_f:\P^1_{\overline{K}}\rightarrow\R$ attached to $f$ is defined by 
$$\hat{h}_f(x)=\lim_{n\rightarrow\infty}\frac{1}{d^n}h(f^n(x)).$$
It is often useful to decompose canonical height attached to rational maps on $\P^1$ in terms of local height function. That is, \begin{equation}\hat{h}_f(x)=\frac{1}{\deg(x)}\sum_{v\in M_K}\sum_{x'\in \mathrm{Gal}(\overline{K}/K)\cdot x}N_v\lambda_{f,v}(x')\notag\end{equation} where the local height $\lambda_{f,v}$ extends continuously to $\A^1_{\C_v}$ and it grows logarithmically, that is, $$\lambda_{f,v}(x)=\log|x|_v+O(1)$$ as $|x|_v\rightarrow\infty$. 
For each $v\in M_K$, there is a notion of distributional valued Laplacian operator $\Delta$ on the Berkovich projective line $\P^{1,\mathrm{an}}_{\C_v}$. The reader may refer to \cite{Ben19, Ber90, BR10} for more background on the  Berkovich projective space $\P^{1,\mathrm{an}}_{\C_v}$. The Laplacian operator is normalized by $$\Delta \log^+|x|_v=\rho_v-\delta_{\infty}$$ on $\P^{1,\mathrm{an}}_{\C_v}$ where $\rho_v=\mu_{S^1}$ is the Lebesgue measure on the unit circle ($v$ archimedean) and $\rho_v=\delta_{\mathrm{Gauss}}$ is the Dirac measure supported at Gauss point on $\P^{1,\mathrm{an}}_{\C_v}$ ($v$ non-archimedean). Note that $\log^+|x|_v$ extends continuously from $\P^1_{\C_v}$ to a real-valued function on $\P^{1,\mathrm{an}}_{\C_v}$. Analogously, for each $v\in M_K$, the local function  $\lambda_{f,v}$ extends naturally to a continuous and subharmonic function on $\A^{1,\mathrm{an}}_{\C_v}$ with logarithmic singularity at $\infty$. Then we have, on $\P^{1,\mathrm{an}}_{\C_v}$,\begin{equation}\Delta\lambda_{f,v}=\mu_{f,v}-\delta_{\infty}\label{eq:laplacian}\end{equation} where $\mu_{f,v}$ is the canonical probability measure associated to rational function $f$ at place each $v$ of $K$. For each place $v\in M_K$, the probability measure $\mu_{f,v} $ satisfies the invariance property $f^*\mu_{f,v}=d\mu_{f,v}.$ For non-archimedean $v\in M_K$, if $f$ have good reduction, the measure $\mu_{f,v}$ is the Dirac measure supported on Gauss point $\zeta_{0,1}$ on $\P^{1,\mathrm{an}}_{\C_v}$.\\ 

The canonical height associated to rational maps enjoys several interesting properties which we record below.
\begin{prop}\label{thm:propertiesofcanoheight} Let $f$ be a rational map of degree $\geq 2$ defined over a number field $K$.
 Then there exist constants $c_1,c_2,c_3,$ and $c_4$, depending only on $d$ such that for any $x\in \P^1_{\overline{K}}$ the following holds:
\begin{enumerate}
    \item[(1)] $|h(f(x))-(\deg f)h(x)|\leq c_1h(f)+c_2$.
    \item[(2)] $|\hat{h}_f(x)-h(x)|\leq c_3h(f)+c_4$.
    \item[(3)]  $\hat{h}_f(f(x))=(\deg f)\hat{h}_f(x)$.
    \item[(4)] $x$ is $f$-preperiodic if and only if $\hat{h}_f(x)=0.$
\end{enumerate}   
Here $h(f)$ denotes height of the polynomial $f$ (see \cite[\S1.6]{BG06}).
\end{prop}
\begin{proof}
For our purpose, we state only for number fields. The above properties (1), (2), and (3) also hold over global field whose proof can be found in \cite[\S B.2 and \S B.4]{HS00} and \cite[\S 3.4]{Sil07}. 
\end{proof}

 For an endomorphism $F=(f_1,f_2,...,f_n)$ on $(\P^1)^n$,
we  define the height $\hat{h}_{F}$ for any point $(x_1,x_2,...,x_n)\in (\P^1)^n$ by \begin{equation}\hat{h}_{F}((x_1,...,x_n)):=\hat{h}_{f_1}(x_1)
+...+\hat{h}_{f_n}(x_n).\label{eq:heightofsplitrat}\end{equation}


\begin{prop}\label{prop:smallpointsproduvar}
Let $F=(f_1,f_2,...,f_n)$ be an endomorphism on $(\P^1)^n$ where $F(x_1,x_2,...,x_n)=(f_1(x_1),f_2(x_2),...,f_n(x_n))$ for rational functions $f_i$ defined over a number field $K$ of degree $\deg(f_i)=d\geq 2$ for all $1\leq i\leq n$. Suppose that $\{X_m:=(x_{1,m},x_{2,m},...,x_{m,n})\}_{m\geq1}$ is a generic sequence in $(\P^1_{\overline{K}})^n$ such that 
$F^m(X_m)=R(X_m)$ for all $m\geq 1$ where $R=(r_1,r_2,...,r_n)$ is any endomorphism on $(\P^1)^n$ defined over $K$. Then
$$\lim_{m\rightarrow\infty}\hat{h}_{F}(X_m)=0.$$
\end{prop}
\begin{proof} 
Since $R(X)$ is an arbitrary endomorphism on $(\P^1)^n$, we have $$R(x_1, \dots, x_n) = (r_1(x_{\sigma(1)}) , \dots, r_n(x_{\sigma(n)})),$$
where $\sigma$ is a permutation of the set $\{1,2, \dots, n\}$.

Notice that for any rational function $f$, $g$, where $\deg(f) > 1$ and $\deg(g) \geq 1$, and an arbitrary point $x \in \P^1_{\overline{K}}$, we have
\begin{align*}
    |\hat{h}_f(g(x))-(\deg g)\hat{h}_{f}(x)|&\leq |\hat{h}_f(g(x))-h(g(x))|+|h(g(x))-(\deg g)h(x)|\\&+(\deg g)|h(x)-\hat{h}_f(x)|\\&\leq C
\end{align*}
where the constant $C$ is independent of $x$, but depending on the data of $f$ and $g$ as a result of Theorem \ref{thm:propertiesofcanoheight}(1)--(2). Thus we have shown that  $$\hat{h}_{f}(g(x))=(\deg g)\hat{h}_{f}(x)+O(1).$$ 

Now, for any index $i \in \{1,2, \dots, n\}$, we denote the period of $i$ under the action of $\sigma$ as $k$. Then we have the assumption in the statement of the proposition implies the following:
\begin{align*}
f^m_{i}(x_{i,m}) &= r_i(x_{\sigma(i), m}),\\
    f^m_{\sigma(i)}(x_{\sigma(i),m}) &= r_{\sigma(i)}(x_{\sigma^2(i), m}),\\
    &\vdots\\
       f^m_{\sigma^{k-1}(i)}(x_{\sigma^{k-1}(i),m}) &= r_{\sigma^{k-1}(i)}(x_{i, m}).
\end{align*}
These will then imply the following relations between canonical heights of points:
\begin{align*}
  d^{m}_i \hat{h}_{f_i}(x_{i,m}) &= d_{r_i}\hat{h}_{f_i}(x_{\sigma(i),m}) + D_1\\
    d^{m}_{\sigma(i)} \hat{h}_{f_i}(x_{{\sigma(i)},m}) &= d_{r_{\sigma(i)}}\hat{h}_{f_i}(x_{\sigma^2(i),m}) + D_2\\
    &\vdots\\
      d^{m}_{\sigma^{k-1}(i)} \hat{h}_{f_i}(x_{{\sigma^{k-1}(i)},m}) &= d_{r_{\sigma^{k-1}(i)}}\hat{h}_{f_i}(x_{i,m}) + D_k,
\end{align*}
where $D_i$, $i \in \{1,2,\dots, k\}$ are some constants only depending on $F$ and $R$, $d_{i} = \deg(f_{i})$ and $d_{r_j} = \deg(r_j)$ for $j \in \{1,2, \dots, k\}$. Then, we eliminating the terms $\hat{h}_{f_i}(x_{\sigma^j(i),n})$ for $j \in \{1, \dots, k-1\}$ and obtains 
\begin{equation}
    (d^m_i - K_{1,m})\hat{h}_{f_i}(x_{i,m}) = K_0 + K_{2,m},
\end{equation}
where $K_0$, $K_{1,m}$ and $K_{2,m}$ are constants such that $K_0$ doesn't depend on $m$ and $\displaystyle\lim_{m \to \infty} K_{1,m} = \displaystyle\lim_{m \to \infty}K_{2,m} = 0$. This implies that $$ \lim_{m \to \infty} \hat{h}_{f_i}(x_{i,m}) = 0.$$


Since the argument holds for arbitrary $i \in \{1,2, \dots,n\}$, by the definition of canonical height of $F$ (\ref{eq:heightofsplitrat}) it is easy to see that 
$$\hat{h}_{F}(X_m)=\hat{h}_{f_1}(x_{1,m})+\hat{h}_{f_2}(x_{2,m})+...+\hat{h}_{f_n}(x_{n,m})
\rightarrow 0$$ as $m\rightarrow\infty$.
\end{proof}
\subsection{A free semigroup criterion for endomorphisms  on $(\P^1)^n$}
\begin{lem}\label{lem: preperiodic-stable-composition}
    Suppose $f$ and $g$ are rational functions defined over $\C$ such that $V = \Prep(f) = \Prep(g)$ and $\deg(f), \deg(g) >1$. Then for any $h \in \langle f, g \rangle$, the semigroup generated by $f$ and $g$ under composition, we have $\Prep(h) = V$.
\end{lem}
\begin{proof}
    Let $K$ be the finitely generated field extension of $\Q$ over which $f$ and $g$ are defined. We prove the lemma by induction on the transcedental degree of $K$ over $\Q$. 
    
    The base case is that $K$ is a number field. Then, for any finite field extension $L$ of $K$, we have $V(L)$ is a finite set and $$f(V(L)) \subseteq V(L),$$
    $$ g(V(L)) \subseteq V(L).$$
    This implies that $h(V(L)) \subseteq V(L)$ and so $V(L) \subseteq \Prep(h)(L)$. Thus $\bigcup_{L \text{ s.t. } [L:K] < \infty} V(L) \subseteq \Prep(h)$, which implies that $h$ and $f$ share infinitely many preperiodic points. Then the unlikely intersection \cite[Theorem 1.2]{BD11}, implies $\Prep(h) = V$.

    Now, suppose $K  = K'(Z)$, where $K'$ is a finitely generated field of transcendental degree not greater than $m$ over $\Q$ and the statement holds true for any finitely generated field of transcendental degree not greater than $m$ over $\Q$. 
    
    We first suppose that $h$ is not isotrivial. Take an aribitrary $p(Z) \in V$, then for any specialization of $Z = z \in K'$, we have $p(z)$ is in $\Prep(f_z) = \Prep(g_z)$ and the induction hypothesis implies $p(z) \in \Prep(h_z)$. Thus $\hat{h}_{h_z}(p(z)) = 0$ for any $z \in K'$ and the specialization \cite[Theorem 4.1]{CS93} implies $\hat{h}_{h}(p(Z)) = 0$. Since $h$ is not isotrivial, we have by \cite[Theorem 1.2]{DL16} that $p(Z)$ is a preperiodic point of $h$. Thus $V \subseteq \Prep(h)$ and the unlikely intersection argument by enlarging $K$ to some finite extensions, again, implies that $\Prep(h) = V$.

    Suppose $h$ is isotrivial, that is there exists an automorphism $\sigma$ over a finite extension of $K$ such that $h$ is conjugated to $h'$ defined over a finite extension $L'$ of $K'$. Take any $p \in V$, we have $\sigma(p)$ is a preperiodic points for both $f' = \sigma \circ f  \circ \sigma^{-1}$ and $g' = \sigma \circ g \circ \sigma^{-1}$. For any specialization $Z = z \in L'$, we have $\sigma(p)(z) \in \Prep(f'_z)= \Prep(g'_z)$ and the induction hypothesis implies $\sigma(p)(z) \in \Prep(h')$. Let $\{X_1, \dots, X_r\}$ be a transcendental basis of $L'$ over some number field $L''$, where $r$ is a positive integer smaller or equal to $m$. Then for any specialization $X_i = x_i \in L''$ for each $i \in \{1,2, \dots, r\}$, we have $\sigma(p)(x_1, \dots, x_r, z) \in \Prep(h'_{x_1, \dots, x_r})(L'')$ for any $z \in L''$. Since $\Prep(h'_{x_1, \dots, x_r})(L'')$ is finite for a Zariski dense set of $(x_1, \dots, x_r)$ in $(L'')^{r}$, we have $\sigma(p)(x_1, \dots, x_r)(z)$ is a rational function only taking finitely many values when $z \in L''$, which implies that it is a constant function doesn't depend on $z$, for a Zariski dense set of $(x_1, \dots, x_r)$. Thus, $\sigma(p) \in L'$.
    
    Since $\sigma(p) \in \Prep(h')$ by the induction hypothesis, $h'$ and $f'$ share an infinite set of preperiodic points by the field enlarging argument again. The unlikely intersection \cite[Theorem 1.2]{BD11} implies that $\Prep(h') = \Prep(f')$. So, $\Prep(h) = \Prep(f) = V$. 

    Then the induction will conclude the proof of the lemma.
\end{proof}
\begin{prop} \label{prop:freesemicriterion}
    Let $F = (f_1, \dots, f_n) $ and $G = (g_1, \dots, g_n)$ be splitting endomorphism defined over $\C$ on $(\P^1)^n$ such that $d = \deg(f_i) = \deg(g_j)$ for all $i , j \in \{1,2, \dots, n\}$, where $d \geq 2$ is a positive integer. Suppose $\Prep(F) = \Prep(G)$, then the semigroup generated by $F$ and $G$ under composition is not free.
\end{prop}
\begin{proof}
    Let $P_m$ be the projection from $(\P^1)^n$ to the first $m$ coordinates, i.e., $(\P^1)^m$. Notice that the composition is compatible with $P_m$ in the sense that $P_m \circ h = P_m(h) \circ P_m $, where $h \in \langle F, G \rangle$ and $P_m(h)$ is the first $m$ coordinate of $h$. We prove the proposition by applying $P_m$ and induction on $m$ for each integer $m \in \{1,2, \dots, n\}$. 

    The base case is that $\langle P_1 (F), P_1  (G) \rangle$ is not free. This is true since $\Prep(F) = \Prep(G)$ implies $\Prep(f_1) = \Prep(g_1)$ and the Tits' alternative of rational functions \cite[Corollary 4.11]{BHPT} implies that $\langle f_1, g_1 \rangle$ is not free. 

    Now, we suppose that $\langle P_m(F), P_m(G) \rangle $ is not a free semigroup for some $m \in \{1,2, \dots, n-1\}$, we will show that $P_{m+1}(F)$ and $P_{m+1}(G)$ also don't generate a free semigroup. Notice that there exists two distinct sequence of compositions $h_1$, $h_2 \in \langle F, G \rangle$, such that 
    $$ P_m(h_1) = P_m(h_2).$$
    We want to show $\Prep(h_{1,m+1}) = \Prep(h_{2, m+1})$, where $h_{i, m+1}$, $i = 1, 2$, are the $m+1$-th coordinate of $h_i$. This follows immediately from the Lemma \ref{lem: preperiodic-stable-composition} since $\Prep(f_{m+1}) = \Prep(g_{m+1})$.
    Now, applying the Tits' alternative of rational functions \cite[Corollary 4.11]{BHPT} again, we have that $ \langle h_{1,m+1}, h_{2,m+1}\rangle$ is not free. Then since $P_m(h_1) = P_m(h_2)$, there exists two distinct sequence of compositions $h'_1$, $h'_2 \in \langle h_1, h_2 \rangle \subseteq \langle F, G \rangle$, such that 
    $$ P_{m+1}(h'_1) = P_{m+1}(h'_2).$$
    Therefore, by the induction, we have $\langle F, G \rangle$ is not a free semigroup.
\end{proof}
\subsection{Common solutions of iterated endomorphisms on $(\P^1)^n$}
\begin{thm}\label{thm:arithmeticequidistribution} (Arithmetic Equidistribution on $\P^1$) \cite{BR06,CL06, FRL06}
Let $K$ be a number field and $f(x)\in K(x)$ be a rational function of degree $d\geq 2$ defined over $K$. Suppose that $\{x_n\}_{n\geq1}$ is an infinite sequence of distinct points in $\P^1_{\overline{K}}$ satisfying $\hat{h}_f(x_n)\rightarrow0$. Then the discrete probability measure supported equally on the Galois conjugates of $x_n$
$$[x_n]_v=\frac{1}{[K(x_n):K]}\sum_{\sigma:K(x_n)\hookrightarrow\C_v}\delta_{\sigma(x_n)}$$ converges weakly to the equilibrium measure $\mu_{f,v}$ in  $\P^{1,\mathrm{an}}_{\C_v}$ at each place $v\in M_K$ in the sense that 
$$\int\varphi d[x_n]_v\rightarrow\int\varphi d\mu_{f,v}$$ for every bounded continuous function $\varphi : \P^{1,\mathrm{an}}_{\C_v}\rightarrow\R$. Here $\delta_x$ denotes the Dirac measure on $\P^{1,\mathrm{an}}_{\C_v}$ supported at the point $x$. 
 \end{thm}
 Next, we prove our main result which asserts the Zariski-non-density of common solutions of two iterated split endomorphisms on $(\P^1)^n$ defined over a number field.  One key ingredient of proof  relies on an arithmetic equidistribution Theorem \ref{thm:arithmeticequidistribution} of small points.
\begin{proof}[Proof of Theorem \ref{thm:commonzerosofendomorphism}] Let $K$ be the field generated by all coefficients of $F, G,$ and $R$. Then $K$ must be a number field.
Suppose that there exists a generic sequence $\{X_j:=(x_{1,j},x_{2,j},...,x_{n,j})\}_{j\geq 1}\subset (\P^1_{\overline{K}})^n$ of the solutions to  $$F^{k_j}(X_j)=G^{l_j}(X_j)=R(X_j).$$ It should be clear that $k_j$ and $l_j$ increase without bound as $j\rightarrow\infty$. It follows immediately from Proposition \ref{prop:smallpointsproduvar}  that the sequence $\{X_j\}_{j\geq1}$ is of small height with respect to $F$ and $G$. That is, $\hat{h}_{F}(X_j)\rightarrow0$ and $\hat{h}_{G}(X_j)\rightarrow0$ as $j\rightarrow\infty$. By the non-negativity of the canonical height $\hat{h}_{f_i}$ (resp. $\hat{h}_{g_i}$) and the definition of $\hat{h}_{F}$ (resp. $\hat{h}_G$) (\ref{eq:heightofsplitrat}), we have that $$\lim_{j\rightarrow\infty}\hat{h}_{f_i}(x_{i,j})=\lim_{j\rightarrow\infty}\hat{h}_{g_i}(x_{i,j})=0$$ for all $1\leq i\leq n.$
For each $1\leq i\leq n$, the sequence $\{x_{i,j}\}_{j\geq 1}$  is small with respect to both $f_i$ and $g_i$. For the rest of the proof, the same argument applied to all indices $1\leq i\leq n$ simultaneously.  Applying the arithemtic equidistribution, Theorem \ref{thm:arithmeticequidistribution}, 
 we obtain that 
$\mu_{f_i,v}=\mu_{g_i,v}$ on $\P_{\C_v}^{1,\mathrm{an}}$ at all places $v\in M_K$. Then the definition of the Laplacian (\ref{eq:laplacian}) implies the equality $$\lambda_{f_i,v}(x)=\lambda_{g_i,v}(x)+ C_v$$ holds true for all places $v\in M_K$. 
  Summing over all places $v$ of $K$, we obtain 
$$\hat{h}_{f_i}(x)=\hat{h}_{g_i}(x)+C.$$ It is easy to deduce that $C\equiv 0$ by taking $x\in \mathrm{Prep}(f_i)$ and $x\in\mathrm{Prep}(g_i)$.  Hence we have shown that $\hat{h}_{f_i}=\hat{h}_{g_i}$ for all $1\leq i\leq n$. As the canonical height of rational functions  detects their preperiodic points  in the sense of  Proposition \ref{thm:propertiesofcanoheight}(4), it gives
$$\mathrm{Prep}(f_i)=\mathrm{Prep}(g_i)$$ for all $1\leq i\leq n$. Hence $\mathrm{Prep}(F)=\mathrm{Prep}(G)$. As asserted in Proposition \ref{prop:freesemicriterion}, we have that the semigroup generated by $F$ and $G$ are not free. In other words, $F$ and $G$ are compositionally dependent. This is a contradiction to the assumption. Therefore, such Zariski dense sequence of solutions to the equation \ref{eq:mainthmeq} does not exist. The proof is completed.

\end{proof}
\section{Tits' Alternative of Regular Polynomial Skew Products}\label{sec:TitAlternativepolyskew}
\begin{defn}
   Let's denote the Chebyshev polynomial of degree $d \in \N$ as $T_d$.  We call a polynomial defined over $\C$ a special polynomial, or say a polynomial is a special, if it is linearly conjugated over $\C$ to either a power map or $\pm T_d$ for some $d \in \N$.
\end{defn}
\begin{defn}
    Let $f$ and $g$ be two polynomials in $k[x]$, where $k$ is an algebraically closed field. We say $f$ is linearly related to $g$ over $k$ if there exists non-constant linear maps $l_1, l_2 \in k[x]$ such that
    $$ f = l_1 \circ g \circ l_2.$$
\end{defn}
\begin{lem}\label{lem: special-to-non-special}
    For a polynomial $f(x,y) = a_dy^d + \sum^{d-1}_{i = 0}a_i(x) y^i$ defined over $\C$, where $a_d \in \C^*$ and $a_i(x) \in \C[x]$. Suppose there are infinitely many $x_0 \in \C$ such that $f(x_0,y)$ is linearly related to $T_d$ (respectively $y^d$), then $g(x,y)$ is linearly related to $T_d$ (respectively $y^d$) over $\overline{\C(x)}$.
\end{lem}
\begin{proof}
Suppose $f(x,y)$ is not linearly related to $T_d(y)$ or $y^d$ in $\overline{\C(x)}[y]$, then there exists a $i\in \{2,3\}$ such that for any non-constant linear maps $l_1(y)$ and $l_2(y) \in \overline{\C(x)}[y]$, $$l_2 \circ f(x,y) \circ l^{-1}_1(y)$$ doesn't commute with any polynomials in $\overline{\C(x)}[y]$ of degree $i$ (see \cite{RJ23} and \cite{RJ20}). Without loss of generality, we assume that $i = 2$. Let $R = \overline{\C(x)}[ a, a^{-1},b,c, \alpha^{\pm 1}_1, \alpha^{\pm 1}_2, \beta_1,\beta_2]$ be a ring extension of $\overline{\C(x)}$ by adjoining transcedental elements $\alpha_1$, $\alpha_2$, $\beta_1$, $\beta_2$, $a$, $b$ and $c$. Let $g(y) = ay^2 + by + c \in R[y]$ and $l_i(y) = \alpha_i y + \beta_i$ for $i\in\{1,2\}$. Let 
$$ F(x, y) = l_2(f(x, l_1^{-1}(g(y)))) - g(l_2(f(x,l^{-1}_1(y)))),$$
and $F_i \in R$, $i \in \{0,1,\dots, d+ 2\}$, is the coefficients of $y^{i}$ terms in $F(y)$. The above implies that 
$$ I = (F_0, \dots, F_{d+2})$$
is a unit ideal in $R$. Thus, there exists $g_i \in R$, $i \in \{0,1,2 \dots, d+2\}$, such that
\begin{equation}\label{eq: unit-equation}
     \sum^{d+2}_{i = 0} g_i F_i = 1.
\end{equation}
Now, we clear the denominators of $g_i$'s and $F_i$'s as functions on $x$ by multiplying a $r(x) \in \overline{\C(x)}^*$ on both sides and denote $g'_i = g_i r$ and $F'_i = F_i r$ for each $i$'s. 

Now, suppose there is an infinite set $S \subseteq \C$ such that for any $x_0 \in S$ we have $f(x_0, y)$ is conjugated to either $T_d(y)$ or $y^d$ in $\C[y]$. Then there is some linear polynomial $l$ in $\C[y]$ such that $T_2$ or $y^2$ commutes with $l \circ f(x_0,y) \circ l^{-1} $. Thus, $\{F'_i(x_0) : i = 0,1,2, \dots, d+2\} $ doesn't generate an unit ideal in $R$. Therefore, we have that for any $x_0 \in S$, $$\sum^{d+2}_{i = 0} g'_i(x_0) F'_i(x_0) = r(x_0) = 0.$$ This implies that $r(x) \equiv 0$ and gives a contradiction.
\end{proof}
\begin{defn}
    For a polynomial $f \in \C[x]$, the set of non-constant linear polynomial $\mu \in \C[x]$ such that there exists a non-constant linear polynomial $\nu \in \C[x]$ satisfying
    $$\nu \circ f = f \circ \mu$$
    forms a group under compositions. We denote it as $\mathcal{G}(f)$.
\end{defn}
\begin{lem}\label{lem: Ritt-decompos-common-iterates}
    Let $f(x)$ and $g(x) \in \C[x]$ be polynomials of degree $d \geq 2$ that are not linearly related to power maps such that $f (x) =  g \circ \lambda(x) $, and $f(x) = \mu \circ g(x)$ for some non-constant linear polynomials $\mu(x)$ and $\lambda(x)$. Then there exists a non-constant linear polynomial $\phi(x)$ such that 
    $$ \phi \circ f \circ \phi^{-1} = \psi \circ \epsilon x^s h(x^t)$$
    $$ \phi \circ  g \circ \phi^{-1} = \psi \circ x^s h(x^t)$$
    for some non-constant linear polynomial $\psi$, polynomial $h$ which is not a power map, positive integer $t$, $t$-th root of unity $\epsilon$ and non-negative integer $s$. 
\end{lem}
\begin{proof}
    From the condition, we see that 
    $$ \mu^{-1} \circ  g \circ \lambda = g.$$
    By \cite[proof of Lemma 4.2]{PF20} and the fact that $g$ is not linearly related to power maps, there exists a linear polynomial $\phi_1$ and $\phi_2$ such that $ \mu$ and $\lambda$ are conjugated by $\phi_1$ and $\phi_2$ respectively to the maps $x \to \epsilon_i x$, $i \in \{1,2\}$, for some roots of unity $\epsilon_1$, $\epsilon_2$ respectively. Moreover, the order of the $\epsilon_1$ and $\epsilon_2$ are not greater than $d$. Then comparing the coefficients on the both side of equation
    $$ \phi_1 \circ g \circ \phi_2^{-1} \circ \epsilon_2 x = \epsilon_1 \phi_1 \circ g \circ \phi_2^{-1},$$
    we get that
    $$ \phi_1 \circ g \circ \phi_2^{-1} = x^sh(x^t)$$
    for some polynomial $h$ which is not a power map, non-negative integer $s$ and positive integer $t$ such that $\epsilon_i$'s are $t$-th roots of unity. Thus we have
    $$ \phi_2 \circ g \circ \phi^{-1}_2 = \phi_2 \circ \phi_1^{-1} \circ x^sh(x^t) ,$$
    $$ \phi_2 \circ f \circ \phi^{-1}_2 =\phi_2 \circ g \circ \phi^{-1}_2 \circ \epsilon_2 =  \phi_2 \circ \phi_1^{-1} \circ \epsilon_1 x^sh(x^t).$$
    Then the lemma follows immediately.
\end{proof}

\begin{lem}\label{lem: sharing-common-iterates-decompos}
    Let $f(x)$ and $g(x)$ be two non-special polynomials defined over $\C$ of equal degree $d$ and sharing a common Julia set $J$. Then there exists a linear map in $ \phi(x) \in \C[x]$ such that 
    $$ \phi \circ f \circ \phi^{-1} =\epsilon_1R(x),$$
    $$ \phi \circ g \circ \phi^{-1} = \epsilon_2 R(x),$$
    where $R(x) = x^sh(x^{t})$, for some polynomial $h$ and positive integers $t$, $s$ and    $\epsilon_1$ and $\epsilon_2$ are roots of unity of order $t$.
\end{lem}
\begin{proof}
Since $J = J(g)  = J(f)$ and $f$, $g$ are non-special, we have, by \cite{SS95}, that $$f = \sigma_1 \circ p^m,$$
$$ g= \sigma_2 \circ p^m$$
for some linear automorphisms $\sigma_i$, positive integer $m$, $i\in \{1,2\}$ such that $\sigma_i(J) = J$ and polynomial $p(x)$ such that $J(p) = J$. Moreover, we can find a linear map $\phi$, such that $\phi \circ p \circ \phi^{-1} = R'(x)$, where $R'(x) = x^{s'}h'(x^{t})$ for some positive integer $s'$ and $t$ and then $\sigma_i$'s are conjugated by $\phi$ to multiplications by $\epsilon_i$, where $\epsilon_i$ is a $t$-th root of unity for each $i \in \{1,2\}$. Thus, we see 
$f$ and $g$ is of the desired form.

\end{proof}
\begin{lem}\label{lem: uniform-conjugation}
    Let $f(x,y)$ and $g(x,y)$ be polynomials in $\C[x][y]$ having equal degree and also the leading coefficients of them are constants. Suppose that there exists a non-negative integer $s$, a positive integer $t$ and a $t$-th root of unity $\epsilon$ such that for infinitely many $x_0 \in \C$, there is non-constant linear polynomials $\phi_{x_0}$, $
    \psi_{x_0}$ and a polynomial $h_{x_0}$ satisfying
    $$ \phi_{x_0} \circ g(x_0,y) \circ \phi^{-1}_{x_0} = \psi_{x_0} \circ \epsilon y^{s}h_{x_0}(y^{t})$$
    $$ \phi_{x_0} \circ f(x_0,y) \circ \phi^{-1}_{x_0 } = \psi_{x_0} \circ y^{s}h_{x_0}(y^{t}).$$ 
    Then there exists non-constant linear polynomials $\phi(y), \psi(y) \in \overline{\C(x)}[y]$ and polynomial $h(y) \in \overline{\C(x)}[y]$ such that
    $$ \phi \circ g(x,y) \circ \phi^{-1} = \psi \circ \epsilon y^sh(y^t)$$
    $$ \phi \circ f(x,y) \circ \phi^{-1} = \psi \circ y^s h(y^t).$$
\end{lem}
\begin{proof}
    Let's suppose, for the purpose of contradictions, that there doesn't exists such $\psi(y)$, $\phi(y)$ and $h(y)$ satisfying
    \begin{equation}\label{eq: uneq-1}
        \phi \circ g(x,y) \circ \phi^{-1} - \psi \circ \epsilon y^sh(y^t) = 0,
    \end{equation}
    \begin{equation}\label{eq: uneq-2}
        \phi \circ f(x,y) \circ \phi^{-1} - \psi \circ y^s h(y^t) = 0.
    \end{equation}
    Let's write $$\phi(y) = \alpha y + \beta$$ $$ \psi(y) = \delta y + \gamma$$ and $$h_1(y) = \sum^{(d-s)/t}_{i = 0}a_i y^i.$$ Let $R = \overline{\C(x)}[\alpha^{\pm 1}, \delta^{\pm 1}, a^{\pm 1}_{(d-s)/t}, \beta, \gamma, a_0, \dots, a_{d-1}]$. 
    
   Then the assumption implies that the coefficients of the left hand side of Equation (\ref{eq: uneq-1}) and (\ref{eq: uneq-2}) as elements in $R$ generates an unit ideal in $R$. We write $\{f_1, \dots, f_m\} \subseteq R$ as the coefficients from above equations that is viewed as equations of polynomials in $y$. Then this is saying that there exists $g_1, \dots, g_m \in R$ such that 
    $$ \sum^m_{i = 1}f_i g_i = 1. $$
    Now, we clear the denominators of the left hand side as a rational function on $x$ by multiplying a $r(x) \in \overline{\C(x)}^*$, which doesn't have a pole, on both side. Let $f'_i = f_i r$ and $g_i' = g_i r$, $i \in \{1,2, \dots, m\}$. Then we have 
    $$ \sum^m_{i = 1} f'_i g'_i = r.$$
    Suppose there are infinitely many $x_0$ such that the assumption in the statement of the lemma hold. Then the above implies that $r(x_0) = 0$ for infinitely many $x_0$ as $f'_i(x_0)$'s can not form an unit ideal due to the existance of $\phi_{x_0}$, $\psi_{x_0}$ and $h_{x_0}$. Thus $r \equiv 0$. This gives a contradiction.
    
    
   Thus, there exists $\phi$, $\psi$ and $h$ that makes the equations $(\ref{eq: uneq-1})$ and $(\ref{eq: uneq-2})$ hold. 
\end{proof}
For the convenience of readers, we summarize a version of Ritts' decomposition that we will be using in the discussion.
\begin{thm}[Theorem 2.3 of \cite{PF17} or see \cite{EH41}] \label{thm: Ritts-decom}
    Let $A,C, D, B$ be polynomials of degree degrees at least $2$ such that \begin{equation}
        A \circ C = D \circ B
    \end{equation} and 
    \begin{equation}
        \deg(A) = \deg(D), \deg(C) = \deg(B).
    \end{equation}
    
Then there exist a linear polynomials $\mu$ such that 
$$ A = D \circ \mu,$$
$$ C = \mu^{-1} \circ B.$$
\end{thm}

\begin{lem}\label{lem: decide-non-speical-and-special-cases}
    Let $g(x,y)$ be a polynomial in $\C[x][y]$ with constant leading coefficient and $f(x)$ be a polynomial such that $\deg_y(g) = \deg(f)  = d> 1$. Let $g_{i,x}(y) = g(f^i(x),y)$. Suppose there is an infinite set $S \subseteq \Per(f)$ such that for any $x_0 \in S$, 
    $$ g_{n-1,x_0} \circ \dots \circ g_{0,x_0}$$
    is linearly conjugated to a power map or $\pm T_{d^n}$, where $n$ is the period of $x_0$. 
    
    Then there exists non-constant linear polynomial $u_x(y), v_x(y) \in \overline{\C(x)}[y] $ such that $g(x,y) = u_x(y) \circ S_d \circ v_x(y)$, where $S_d$ is either a power map or $\pm T_d$. Moreover, $v_{f(x)} \circ u_x = c  y$ for some constant $c \in \C^*$ if $S_d = y^d$ or otherwise $v_{f(x)} \circ u_x \in \{ \pm y\}$.
\end{lem}
\begin{rmk}
    As maps of the form $g(f^i(x),y)$ will appear a lot in the discussion, we will continue using the notation $g_{i,x}(y) = g(f^i(x),y)$ in the rest of this section without further explanation.
\end{rmk}
\begin{proof}
    For any $x_0\in \Per(f)$, suppose $g_{n-1,x_0} \circ \dots \circ g_{0,x_0}$ is linearly conjugated to a special polynomial. Then by Ritt's decomposition, Theorem \ref{thm: Ritts-decom}, $g(x_0,y)$ is linearly related to a special polynomial. Then, since there are infinitely many such $x_0 \in \Per(f)$, by Lemma \ref{lem: special-to-non-special}, we have that there exists non-constant linear polynomial $u_x(y)$ and $v_x(y) \in \overline{\C(x)}[y]$ such that $g(x,y) = u_x \circ S_d \circ v_x$. 
    
    Moreover, since 
    $$ g_{n-1,x_0} \circ \dots \circ g_{0,x_0}$$
    is linearly conjugated to special polynomials, we can apply Theorem \ref{thm: Ritts-decom} to 
    \begin{align}
        & u_{f^{n-1}(x_0)} \circ S_d \circ v_{f^{n-1}(x_0)} \circ \dots \circ S_d \circ v_{f(x_0)} \circ u_{x_0} \circ S_d \circ v_{x_0} \nonumber \\
        & = l^{-1} \circ S_{d^n}\circ l,
    \end{align}
    for some non-constant linear polynomial $l$, and we get 
    \begin{equation}
        v_{f(x_0)} \circ u_{x_0} \circ S_d \circ v_{x_0} = \mu \circ S_d \circ l
    \end{equation}
    \begin{equation}\label{eq: ritts-4-9-1}
         u_{f^{n-1}(x_0)} \circ S_d \circ v_{f^{n-1}(x_0)} \circ \dots \circ S_d  = l^{-1} \circ S_{d^{n-1}} \circ \mu^{-1}
    \end{equation}
    for some non-constant linear polynomial $\mu \in \C[y]$. 
    Thus, $$\mu^{-1} \circ v_{f(x_0)} \circ u_{x_0} \in \mathcal{G}(S_d),$$
    since $\mathcal{G}(S_d)$ is a group formed by scaling multiplies \cite[Lemma 4.1 and Lemma 4.2]{PF20} and therefore for any non-constant linear map $l_1 \in \mathcal{G}(S_d)$ and $l_2 \in \C[x]$ such that 
    $$ l_2 \circ S_d = S_d \circ l_1,$$
    we also have $l_2 \in \mathcal{G}(S_d)$.
    Now, apply Theorem \ref{thm: Ritts-decom} again to Equation (\ref{eq: ritts-4-9-1}), we would have 
    \begin{equation}
        \mu_1 \circ S_d  = S_d \circ \mu^{-1}
    \end{equation}
    for some non-constant linear polynomial $\mu_1 \in \C[y]$. Thus, $\mu \in \mathcal{G}(S_d)$ and we also have 
    $$ v_{f(x_0)} \circ u_{x_0} \in \mathcal{G}(S_d).$$

     By \cite[Proof of Lemma 4.2]{PF20}, we know that if $S_d = \pm T_d$, then $$v_{f(x_0)} \circ u_{x_0} \in \mathcal{G}(\pm T_d) = \{ \pm y\}.$$
      Since this holds for infinitely many $x_0 \in \C$, we have $v_{f(x)} \circ u_x \in \{ \pm y\}$. 
    
    Now if $S_d = y^d$, then $v_{f(x_0)} \circ u_{x_0} = c_{x_0} y$ for some $c_{x_0} \in \C^*$ \cite[Lemma 4.1]{PF20}. But, notice that if $S_d = y^d$, since $g(x,y)$ has constant leading coefficients as a polynomial in $y$ by the assumption, we can always take $v_x(y)$ and $u_x(y)$ such that they have constant leading coefficients. Then in this case, 
    $$ v_{f(x_0)} \circ u_{x_0} = c y$$
    for a constant $c$ for infinitely many $x_0 \in \Per(f)$, which implies 
    $$v_{f(x)} \circ u_x = cy.$$

\end{proof}

\begin{prop}\label{prop: key-prop-alternative}
     Let $F(x,y) = (f(x), g_1(x,y))$ and $G(x,y) = (f(x), g_2(x,y))$ be polynomial skew products defined over $\C$ such that 
    $\deg(f_1) = \deg(f_2) = \deg_y(g_1) = \deg_y(g_2)$ and $g_1(x,y)$ and $g_2(x,y)$ have constant leading coefficients as a polynomial in $\C[x][y]$. 
    Suppose that for any $x_0 \in \Per(f)$ we have 
    \begin{equation}
         \Prep(g_{1,n-1,x_0} \circ \dots \circ g_{1,0,x_0}) = \Prep(g_{2,n-1,x_0} \circ \dots \circ g_{2,0,x_0})
    \end{equation}
    where $n$ is the period of $x_0$.
    Then the semigroup generated by $F(x,y)$ and $G(x,y)$ under composition is not free.
\end{prop}
\begin{proof}
By assumption, we have that for any $x_0 \in \Per(f)$, we have $$ \Prep(g_{1,n-1,x_0} \circ \dots \circ g_{1,0,x_0}) = \Prep(g_{2,n-1,x_0} \circ \dots \circ g_{2,0,x_0}),$$
where $n$ is the period of $x_0$.\newline \newline 

 \noindent\textbf{Case I.} Now, suppose there are infinitely many $x_0 \in \Per(f)$ such that either $$G_{2,x_0} = g_{2,n-1,x_0} \circ \dots \circ g_{2,0,x_0}$$ or $$G_{1,x_0}= g_{1,n-1,x_0} \circ \dots \circ g_{1,0,x_0}$$ is linearly conjugated by a non-constant linear map $\phi_{x_0 } \in \C[y]$ to a special map. Without loss of generality, we assume that it is $G_{1,x_0}$. 

On the other hand, since $\Prep(G_{2,x_0}) = \Prep(G_{1,x_0})$, we have $G_{2,x_0}$ is also conjugated by the same non-constant polynomial $\phi_{x_0}$ to $ \xi S_{d^n}$ for some root of unity $\xi$. Now, by Lemma \ref{lem: decide-non-speical-and-special-cases} there exists a linear polynomial $u_{x}, v_x \in \overline{\C(x)}[y]$ with constant leading coefficients such that 
$$g_1(x,y)  = v_{x} \circ S_d\circ u_{x},$$
and 
$$ u_{f(x)} \circ v_{x} = c_1y$$
for some $c_1 \in \C^*$ if $S_d = y^d$ and $u_{f(x)} \circ v_{x}$ in a finite group otherwise. The same argument can be applied to $g_2(x,y)$ and $G_{2,x_0}$ to obtain linear polynomials $u'_{x}, v'_x \in \overline{\C(x)}[y]$ with constant leading coefficients such that 
$$g_2(x,y) =  v'_{x} \circ S_d \circ u'_{x} ,$$
$$ u'_{f(x)} \circ v'_x = c_2 y,$$
for some $c_2 \in \C^*$ if $S_d = y^d$ and $u_{f(x)} \circ v_{x}$ in a finite group otherwise.

\noindent\textbf{Subcase \RNum{1}.}
We first suppose $S_d= y^d$.
Then we further have, 
$$ v^{-1}_{f^{n-1}(x_0)} \circ G_{1,x_0} \circ u^{-1}_{x_0} = c^{(d- d^{n})/(1-d)}_1 y^{d^{n}}$$
$$ v'^{-1}_{f^{n-1}(x_0)} \circ G_{2,x_0} \circ u'^{-1}_{x_0} =  c^{(d- d^{n})/(1-d)}_2 y^{d^{n}}.$$

From these, we have 
\begin{equation*}
    u_{x_0} \circ G_{1,x_0} \circ u^{-1}_{x_0} = c^{(d- d^{n})/(1-d)+1}_1 y^{d^{n}},
\end{equation*}
\begin{equation*}
     u'_{x_0} \circ G_{2,x_0} \circ u'^{-1}_{x_0} = c^{(d- d^{n})/(1-d)+1}_2 y^{d^{n}}.
\end{equation*}

Since $\Prep(G_{1,x_0}) = \Prep(G_{2,x_0})$, we have 
\begin{equation}\label{eq: thm-4-11-prep-conj-eq}
    \Prep( u'_{x_0} \circ G_{1,x_0} \circ u'^{-1}_{x_0}) = \Prep( u'_{x_0} \circ G_{2,x_0} \circ u'^{-1}_{x_0}).
\end{equation}
Notice that 
\begin{equation}\label{eq: thm-4-11-c_3-1}
    u'_{x_0} \circ G_{2,x_0} \circ u'^{-1}_{x_0} = c_2^{(d-d^n)/(1 - d)+1} y^{d^n},
\end{equation}
since $u'_{x_0} \circ v'_{f^{n-1}(x_0)} = c_2 y$. Thus, there exits a linear polynomial $\mu (y) = c_3 y$ for a constant $c_3 \in \C^*$ such that 
\begin{equation}\label{eq: thm-4-11-c_3-2}
    \mu \circ  u'_{x_0} \circ  G_{2,x_0} \circ u'^{-1}_{x_0} \circ \mu^{-1} = y^{d^n}.
\end{equation}
Also, 
\begin{equation}\label{eq: thm-4-11-eq-root-1}
        u'_{x_0} \circ G_{1,x_0} \circ u'^{-1}_{x_0} \ = u'_{x_0} \circ v_{f^{n-1}(x_0)} \circ c^{(d- d^{n})/(1-d)+1}_1 y^{d^n} \circ u_{x_0} \circ u'^{-1}_{x_0}.
\end{equation}
Now, Equation (\ref{eq: thm-4-11-prep-conj-eq}) implies that 
\begin{equation}\label{eq: thm-4-11-eq-root-2}
     \mu \circ u'_{x_0} \circ G_{1,x_0} \circ u'^{-1}_{x_0} \circ \mu^{-1}= \xi_1 y^{d^n}
\end{equation}
for some root of unity $\xi_1$.
Now, put together Equation (\ref{eq: thm-4-11-eq-root-1}) and (\ref{eq: thm-4-11-eq-root-2}), we have 
\begin{equation*}
   \mu \circ  u'_{x_0} \circ u^{-1}_{x_0} \circ c_2 \cdot c^{(d- d^{n})/(1-d)}_1 y^{d^{n}} \circ u_{x_0}\circ u'^{-1}_{x_0} \circ \mu^{-1}= \xi_1 y^{d^{n}},
\end{equation*} 
since $u^{-1}_{x_0} \circ c_2y =v_{f^{n-1}(x_0)} $. This is equivalent to 
\begin{equation}\label{eq: main-1}
    u'_{x_0} \circ u^{-1}_{x_0} \circ c_2 \cdot c^{(d- d^{n})/(1-d)}_1 y^{d^{n}} \circ u_{x_0}\circ u'^{-1}_{x_0}  = \xi_1 c_2^{(d-d^n)/(1 - d) +1} y^{d^n}
\end{equation}
by Equations (\ref{eq: thm-4-11-c_3-1}) and (\ref{eq: thm-4-11-c_3-2}).
Thus, $u_{x_0} \circ u'^{-1}_{x_0}$ is a multiplication by constant for infinitely many $x_0 \in \Per(f)$. Therefore, \begin{equation}\label{eq: thm-4-11-uv-1}
    u_x \circ u'^{-1}_x = c_4 y
\end{equation} for some $c_4\in \C^*$ as $u_x$ and $u'_x$ both have constant linear coefficients.

Now, since Equation (\ref{eq: main-1}) holds for infinitely many $x_0 \in \Per(f)$, we have that there are infinitely many $n \in \N$ such that
\begin{equation}\label{eq: thm-4-11-crelat-1}
    c_4^{d^{n} - 1} = \xi_{1} (c_2/c_1)^{(d - d^{n})/(1 - d) }.
\end{equation}
From Equation (\ref{eq: thm-4-11-uv-1}), we also get that 
\begin{equation*}
    v_x' \circ c^{-1}_4c^{-1}_2y = v_x \circ c^{-1}_1 y.
\end{equation*}
Thus, for any $n \in \N$ we have
\begin{equation*}
    F^n (x,y)= (f^n(x) , v_{f^{n-1}(x)} \circ c_1^{(d-d^n)/(1 - d)}y^{d^n} \circ u_x)
\end{equation*}
\begin{align}
    G^n(x,y) & = (f^n(x) , v'_{f^{n-1}(x)} \circ c_2^{(d-d^n)/(1 - d)} y^{d^n} \circ u'_x ) \\
    & = (f^n(x), v_{f^{n-1}(x)} \circ c_4 c^{-1}_1 c_2^{(d-d^n)/(1 -d)+1} y^{d^n} \circ c^{-1}_4 u_x) \\
    & = (f^n(x), v_{f^{n-1}(x)} \circ c^{- d^n + 1}_4 c^{-1}_1 c_2^{(d-d^n)/(1 -d)+1} y^{d^n} \circ u_x)
\end{align}
By Equation (\ref{eq: thm-4-11-crelat-1}), we have there exists a $n_0 \in \N$ such that
\begin{equation*}
    G^{n_0}(x,y) = (f^{n_0}(x) , v_{f^{n_0-1}(x)} \circ \xi^{-1}_1 c_1^{(d-d^{n_0})/(1 - d)}y^{d^{n_0}} \circ u_x).
\end{equation*}
Notice that there exists a pair of positive integers $n_1$, $n_2$ such that $n_1 > n_2$ and
\begin{equation*}
    (\xi^{-1}_1 y^{d^{n_0}})^{ n_1} = (\xi^{-1}_1y^{d^{n_0}})^{ n_2} \circ (y^{d^{n_0}})^{ (n_1 - n_2)}.
\end{equation*}
Therefore, we have that 
\begin{equation*}
    G^{n_0n_1} = G^{n_0n_2} \circ F^{n_0(n_1 - n_2)}.
\end{equation*}

This concludes that the semigroup generated by $F$ and $G$ are not free. \newline 

\noindent\textbf{Subcase \RNum{2}.} Now, we suppose $S_d = \pm T_d$. We would have similarly
$$ v^{-1}_{f^{n-1}(x_0)} \circ G_{1,x_0} \circ u^{-1}_{x_0} = c_1s_1 T_{d^{n}} = u_{x_0} \circ G_{1,x_0} \circ u^{-1}_{x_0} ,$$

$$ v'^{-1}_{f^{n-1}(x_0)} \circ G_{2,x_0} \circ u'^{-1}_{x_0} =  c_2s_2 T_{d^{n}} =  u'_{x_0} \circ G_{2,x_0} \circ u'^{-1}_{x_0},$$
for infinitely many periodic point $x_0$ of $f$, where $s_1,s_2, c_1, c_2 \in \{\pm 1\}$ and $n$ is the period of $x_0$. 
Thus, 
\begin{equation*}
    G_{1,x_0} = u^{-1}_{x_0} \circ c_1 s_1 T_{d^n} \circ u_{x_0},
\end{equation*}
\begin{equation*}
    G_{2,x_0} = u'^{-1}_{x_0} \circ c_2 s_2 T_{d^n} \circ u'_{x_0}.
\end{equation*}
Since we have that $\Prep(G_{1,x_0}) = \Prep(G_{2,x_0})$, we have that 
\begin{equation*}
\Prep(u'_{x_0} \circ G_{1,x_0} \circ u'^{-1}_{x_0}) = \Prep(u'_{x_0} \circ G_{2,x_0} \circ u'^{-1}_{x_0})
    \end{equation*}
which implies that 
\begin{equation*}
    \Prep(u'_{x_0} \circ u^{-1}_{x_0} \circ c_1s_1 T_{d^n} \circ u_{x_0}\circ u'^{-1}_{x_0}) = \Prep(c_2s_2T_{d^n}).
\end{equation*}
Notice that this gives that 
\begin{equation*}
    u'_{x_0} \circ u^{-1}_{x_0} \in \{\pm y\}.
\end{equation*}
Therefore, we also have 
\begin{equation*}
    v^{-1}_{x_0} \circ v'_{x_0} \in \{\pm{y}\}.
\end{equation*}
Since the above are satisfied for infinitely many $x_0 \in \Per(f)$, we have 
\begin{equation*}
    u'_{x} \circ u^{-1}_{x} \in \{\pm y\},
\end{equation*}
\begin{equation*}
    v^{-1}_{x} \circ v'_{x} \in \{\pm{y}\}.
\end{equation*}
Therefore, 
\begin{equation*}
    F(x,y) = (f(x), v_x \circ S_d \circ u_x)
\end{equation*}
\begin{equation*}
    G(x,y) = (f(x),  v_x \circ sS_d \circ u_x)
\end{equation*}
for some $s \in \{\pm 1\}$. Thus, there exists a pair of positive integer $n_1$, $n_2$ such that 
\begin{equation*}
    G^{n_1 + n_2} = G^{n_1} \circ F^{n_2}.
\end{equation*}
We then conclude that the semigroup generated by $F$ and $G$ is not free in this case. \newline 
\newline

 \noindent\textbf{Case II.} Now, suppose there are only finitely many $x_0 \in \Per(f)$, such that at least one of $G_{1,x_0}$ and $G_{2,x_0}$ are linearly conjugated to special maps. Then we can take an infinite subset $S \subseteq \Per(f)$ such that for any $x_0 \in \Orb_f(S)$, $G_{1,x_0}$ and $G_{2,x_0}$ are non-special. Then, by a series of detailed studies on the Julia set of polynomials \cite{BA90}, \cite{BA92}, \cite{BE87}, \cite{SS95}, we have $G_{1,x_0} = \sigma \circ G_{2,x_0}$, where $\sigma$ is a non-constant linear polynomial defined over $\C$ fixing the Julia set $J(G_{2,x_0}) = J(G_{1,x_0})$. 

Now, Theorem \ref{thm: Ritts-decom} implies that 
\begin{equation}\label{eq: thm-4-11-case-2-1}
    \mu_{x_0} \circ g_{1,0,x_0} = g_{2,0,x_0},
\end{equation}
for some linear polynomial $\mu_{x_0} \in \C[y]$.
Also, notice that $$G_{3,x_0} = g_{1, n, x_0} \circ g_{1,n-1,x_0} \circ \dots \circ g_{1,1,x_0} $$ and 
$$ G_{4,x_0} =g_{2, n, x_0} \circ g_{2,n-1,x_0} \circ \dots \circ g_{2,1,x_0}  $$ 
are non-special as well. Furthermore, $$\Prep(G_{3,x_0}) = \Prep(G_{4,x_0})$$ and $$ g_{1, n, x_0} = g_{1,0,x_0},$$
$$ g_{2, n, x_0} = g_{2,0,x_0}.$$
Thus, we have, similarly, 
$$G_{3,x_0} = \sigma_1 \circ G_{4,x_0}$$
for a non-constant linear polynomial $\sigma_1$ fixing the Julia set of $G_{4,x_0}$ and hence
\begin{equation} \label{eq: left-Ritts-decomp}
    g_{1,0,x_0} \circ \nu_{x_0} = \sigma_{1} \circ g_{2,0,x_0},
\end{equation}
for some linear polynomial $\nu_{x_0} \in \C[y]$ by Theorem \ref{thm: Ritts-decom}. Now, by Lemma \ref{lem: sharing-common-iterates-decompos}, we have in particular that there exists a non-constant linear polynomial $\sigma'$ such that 
$$ \sigma_1 \circ G_{4,x_0} = G_{4,x_0} \circ \sigma' .$$
Thus, Theorem \ref{thm: Ritts-decom} implies that there exists a non-contant linear polynomial $\sigma''$ such that
$$ \sigma_1 \circ g_{2,n,x_0} = g_{2,n,x_0} \circ \sigma''.$$ 
Now, we can abuse the notation and write Equation (\ref{eq: left-Ritts-decomp}) simply as 
\begin{equation}\label{eq: thm-4-11-case-2-2}
    g_{1,0,x_0} \circ \nu_{x_0} = g_{2,0,x_0},
\end{equation} 
for some non-constant linear polynomial $v_{x_0}$ by combining two linear maps $v_{x_0}$ and $\sigma''$.
\newline

 \noindent\textbf{Subcase 1.} Suppose there are infinitely many $x_0 \in S$ such that both of $g_1(x_0,y)$ and $g_2(x_0,y)$ are not linearly related to power maps. Lemma \ref{lem: Ritt-decompos-common-iterates} implies that for infinitely many $x_0 \in S$ there exists non-constant linear polynomials $\phi_{ x_0}(y), \psi_{x_0}(y) \in \C[y]$, a polynomial $h_{x_0}(y) \in \C[y]$ which is not a power map, non-negative integer $s_0$ and positive integer $t_0$ such that
    $$ \phi_{x_0} \circ g_{1,0,x_0} \circ \phi^{-1}_{x_0}  = \psi_{x_0} \circ y^{s_0}h_{x_0}(y^{t_0}) ,$$
    $$\phi_{x_0} \circ g_{2,0,x_0} \circ \phi^{-1}_{x_0} = \psi_{x_0} \circ \epsilon'_0 y^{s_0}h_{x_0}(y^{t_0}) $$
    where $\epsilon'_0$ is $t_0$-th root of unity. 
    Then, since there are only finitely many pairs of integers $(s_0,t_0)$, where $s_0$ is non-negative and $t_0$ is positive satisfying $s_0, t_0 \leq \deg_y(g_1(x,y)),$, by the pigeonhole principle there exists a pair of $(s,t)$ such that for infinitely many $x_0 \in S$ the above hold with $(s,t)$ in place of $(s_0,t_0)$. Now, Lemma \ref{lem: uniform-conjugation} implies that there exists non-constant linear polynomials $\phi(x,y), \psi(x,y) \in \overline{\C(x)}[y]$ such that 
    $$ \phi \circ g_1(x,y) \circ \phi^{-1 } = \psi \circ \epsilon y^sh(x,y^t)$$
    $$\phi \circ g_2(x,y) \circ \phi^{-1 } =  \psi \circ y^sh(x,y^t) $$
    for some $h(x,y) \in \overline{\C(x)}[y]$, non-negative integer $s$, positive integer $t$ and $t$-th root of unity $\epsilon$. 
    
    If $\epsilon = 1$, then $F = G$ and we are done, so we assume that $\epsilon \neq 1$.
Now, $ G_{1,x_0} = \sigma \circ G_{2,x_0}$ implies that 
\begin{align}
     g_{1,n-1,x_0} \circ \dots \circ \phi^{-1}(f(x_0),y) \circ \psi(f(x_0),y) \circ \epsilon y^{s}h(f(x_0),(y^{t}))   \nonumber \\ \circ \phi(f(x_0),y) \circ \phi^{-1}(x_0,y) \circ \psi(x_0,y) \circ  \epsilon y^{s}h(x_0,y^t) \circ \phi(x_0,y)  \nonumber\\
     =\sigma \circ   g_{2,n-1,x_0} \circ \dots \circ \phi^{-1}(f(x_0),y) \circ \psi(f(x_0),y) \circ y^{s}h(f(x_0),(y^{t})) \nonumber \\ \circ \phi(f(x_0),y) \circ \phi^{-1}(x_0,y) \circ \psi(x_0,y) \circ  y^{s}h(x_0,y^t) \circ \phi(x_0,y). 
\end{align}
By Theorem \ref{thm: Ritts-decom} and the right cancellation, this implies \begin{equation}\label{eq: thm-4-11-ggroup-1}
    \phi(f(x_0),y) \circ \phi^{-1}(x_0,y) \circ \psi(x_0,y) \circ \epsilon y \circ \psi^{-1}(x_0,y) \circ \phi(x_0,y) \circ \phi^{-1}(f(x_0),y)  \in \mathcal{G}(y^{s}h(f(x_0),y^{t})),
\end{equation} which is a finite cyclic group. Since we know that $\epsilon y \in \mathcal{G}(y^{s}h(f(x_0),y^{t})) $, we further have that
$$ \phi(f(x_0),y) \circ \phi^{-1}(x_0,y) \circ \psi(x_0,y) = c_1 y$$
is a multiplication by a constant $c_1$ in $\C^*$ since we assumed that $\epsilon \neq 1$.

Notice that this holds for every point in the periodic cycle of $x_0$ by repeating the argument above with $f^i(x_0)$ in place of $x_0$ and therefore, we have 
\begin{equation}
    \phi(f^{i+1}(x_0),y) \circ \phi^{-1}(f^{i}(x_0),y) \circ \psi(x_0,y) = c_{i+1} y
\end{equation}
for some constant $c_{i+1} \in \C^*$ for every positive integer $i$.

Thus, we have for any positive integer $m$, denote
\begin{align}
    G_{1,m,x_0} = g_{1,m,x_0} \circ \dots \circ g_{1,0,x_0}
\end{align}
\begin{align}
     G_{2,m,x_0} = g_{2,m,x_0} \circ \dots \circ g_{2,0,x_0}
\end{align}
and we have

\begin{align}
   & G_{1,m, x_0} = \phi(f^{m+1}(x_0),y)^{-1} \circ c_{m+1} \epsilon y^sh(f^{m}(x_0),y^t) \circ c_{m}\epsilon  y^sh(f^{m-1}(x_0),y^t) \nonumber \\ & \circ \dots \circ c_1\epsilon y^sh(x_0,y^t) \circ \phi(x_0,y),
\end{align}
\begin{align}
    & G_{2,m, x_0} = \phi(f^{m+1}(x_0),y)^{-1} \circ  c_{m+1}  y^sh(f^{m}(x_0),y^t) \circ c_{m} y^sh(f^{m-1}(x_0),y^t) \nonumber \\ & \circ \dots \circ c_1 y^sh(x_0,y^t) \circ \phi(x_0,y).
\end{align}

Thus, there exists a pair of positive integers $n_{s,t}$ and $m_{s,t}$ only depending on $(s,t)$ such that 
\begin{align}
    G_{1,n_{s,t}, f^{m_{s,t}}(x_0)} \circ G_{2,m_{s,t},x_0} = G_{1,n_{s,t}+m_{s,t},x_0}.
\end{align} 
This implies
$$F^{n_{s,t}}(x_0,y) \circ G^{m_{s,t}}(x_0,y) = F^{n_{s,t}+m_{s,t}}(x_0,y). $$
Since it holds for infinitely many $x_0 \in S \subseteq \Per(f)$.
This implies 
$$ F^{n_{s,t}} \circ G^{m_{s,t}} = F^{n_{s,t}+ m_{s,t}}.$$ \newline

     \noindent\textbf{Subcase 2.} Suppose there are at most finitely many $x_0 \in S$ such that both of $g_1(x_0,y)$ and $g_2(x_0,y)$ are not linearly related to a power map. Then there are infinitely many $x_0 \in S$, such that for any $x'_0 \in \Orb_f(x_0)$, we have at least one of $g_1(x_0,y)$ and $g_2(x_0,y)$ is linearly related to a power map.
     
     Notice that a for $x_0 \in S$, without loss of generality, such that $g_1(x_0,y)$ is linearly related to a power map, there exists a non-constant linear polynomials $\psi_{x_0}, \phi_{x_0} \in \C[y]$ such that 
     $$ \psi_{x_0} \circ g_1(x_0,y) \circ \phi_{x_0} = y^d .$$
     Then the fact that $\mathcal{G}(y^d)$ is given by multiplying constants and Equation (\ref{eq: thm-4-11-case-2-1}) and (\ref{eq: thm-4-11-case-2-2}) together imply that
     $$ \psi_{x_0} \circ g_{2}(x_0,y) \circ \phi_{x_0} = c_{x_0}y^d$$
     for some $c_{x_0} \in \C^*$. In particular, $g_2(x_0,y)$ is also linearly related to a power map.

     Thus, we can assume that there is an infinite subset $S' \subseteq S$ such that for any $x_0 \in S'$, we have $g_1(x'_0,y)$ and $g_2(x'_0,y)$ are both linearly related to power maps for any $x'_0 \in \Orb_f(x_0)$. In particular, there exists non-constant linear polynomials $\phi_{x_0}$ and $\psi_{x_0}$ for each $x_0 \in \Orb_f(S')$, such that 
     $$ \psi_{x_0} \circ g_1(x_0,y) \circ \phi_{x_0} = y^d ,$$
     $$ \psi_{x_0} \circ g_{2}(x_0,y) \circ \phi_{x_0} = c_{x_0}y^d.$$

     Again, we assume that $c_{x_0} \neq 1$ except for at most finitely many $x_0 \in \Orb_f(S')$ as otherwise $g_1(x,y) = g_2(x,y)$ and we are done.
     
     Similarly as how we get Expression (\ref{eq: thm-4-11-ggroup-1}) in Subcase \RNum{1} above, $$G_{1,x_0} = \sigma \circ G_{2,x_0}$$ implies, by Ritts' decomposition, Theorem \ref{thm: Ritts-decom}, and the right cancellation, that 
     $$\phi^{-1}_{f(x_0)} \circ \psi^{-1}_{x_0} \circ c_{x_0} y \circ \psi_{x_0} \circ \phi_{f(x_0)}  \in \mathcal{G}(y^d),$$
     which also holds after replacing $x_0$ with $f^i(x_0)$ for any positive integer $i$. Thus, for any positive integer $i$, 
     $$ \phi^{-1}_{f^{i+1}(x_0)} \circ \psi^{-1}_{f^i(x_0)}$$
      is a multiplication by a constant in $\C^*$ for any $x_0 \in \Orb_f(S')$. 
     Thus, we have that there are infinitely many $x_0 \in \Per(f)$ such that 
     \begin{align}
         G_{1,x_0} = \psi^{-1}_{f^{n-1}(x_0)} \circ c_{1,x_0} y^{d^n} \circ \phi^{-1}_{x_0} 
     \end{align}
     \begin{equation}
          = \phi_{x_0} \circ c_{2,x_0} y^{d^n} \circ \phi^{-1}_{x_0},
     \end{equation}
     for some $c_{1,x_0}, c_{2,x_0} \in \C^*$ and $n \in \N$ is the period of $x_0$, which implies that $G_{1,x_0}$ is linearly conjugated to $y^{d^n}$.
     This leads us back to Case \RNum{1} and contradicts our assumption in Case \RNum{2}. So we are done with the proof. 

\end{proof}

\begin{lem}\label{lem: reduce-to-same-fiberation}
    Let $f(y)$, $h(x)$ and $g(x,y)$ be polynomials defined over $\C$, where $g(x,y)$ has constant leading coefficients as a polynomial in $\C[y][x]$. Let $g_i(x,y)$ denote $g(x,f^i(y))$ and $$G_{i,y_0}(x) = g_i(x, y_0) \circ \dots \circ g_0(x,y_0) .$$ Suppose that for any $y_0 \in \Per(f)$, we have 
    $$ \Prep(G_{n_0-1,y_0}) = \Prep(h)$$
    where $n_0$ is the period of $y_0$. We have that $g(x, y)$ doesn't depend on $y$ and $\Prep(g) = \Prep(h)$.
\end{lem}
\begin{proof}
    Let's first suppose that $h$ is not a special polynomial. Then, from \cite{SS95}, for any $y_0 \in \Per(f)$, since $\Prep(G_{n_0-1,y_0}) = \Prep(h)$, we have $$J = J(G_{n_0-1,y_0}) = J(h)$$ and there exists a polynomial $p(x)$ only depending on $J$ and automorphisms $\sigma_1$, $\sigma_2$ preserving their Julia set such that $\Prep(p) = \Prep(h)$ and 
   \begin{equation}\label{eq: 4.12-1}
       G_{n_0-1, y_0} = \sigma_1 \circ p^{m_1},
       \end{equation}
    \begin{equation}
       h = \sigma_2 \circ p^{m_2}.
   \end{equation}
   In particular, let $y_0$ be a fixed point of $f(y)$, we have the above implies that $\deg_x(g) = \deg(p)^{m}$ for some positive integer $m$.
   Then, for any $y_0 \in \Per(f)$ , we have, by Equation (\ref{eq: 4.12-1}), Theorem \ref{thm: Ritts-decom} that there exists a non-constant linear polynomial $\mu_{y_0}(x)$ such that 
   \begin{equation}\label{eq: no-y-depend-1}
       \mu_{y_0} \circ g(x,y_0) = p^m. 
   \end{equation}
  
   Similarly, let $y_1= f(y_0)$, we have, by the same argument as above with $y_1$ in place of $y_0$,
   \begin{equation}
       g_{n_0-1}(x,y_1) \circ \dots \circ g_{0}(x,y_1) = \sigma'_1 \circ p^{m'_1},
   \end{equation}
   for some positive integer $m'_1$ and non-constant linear polynomial $\sigma'_1$ fixing the Julia set of $p$. Notice that $g_{n_0-1}(x,y_1) = g(x,y_0)$. Thus, by Theorem \ref{thm: Ritts-decom} again, we have that there exists a non-constant linear polynomial $\nu_{y_0}(x)$ such that
   \begin{equation}\label{eq: no-y-depend-2}
       g(x,y_0) \circ \nu_{y_0} = \sigma'_1 \circ p^m.
   \end{equation}
   Now, combining Equation (\ref{eq: no-y-depend-1}) and (\ref{eq: no-y-depend-2}), we have $\nu_{y_0} \in \mathcal{G}(p^m)$. 
   
    As $h$ is not a special polynomial, we have $\mathcal{G}(p^m)$ and the set of non-constant linear polynomials fixing the Julia set of $p$ are finite sets. Then, by the pigeonhole principle and Equation (\ref{eq: no-y-depend-2}), there exist non-constant linear polynomials $\nu$ and $\tau$, where $\tau (J(p)) = J(p)$, such that there are infinitely many $y_0 \in \Per(f)$ satisfies
   \begin{equation}
       g(x,y_0) \circ \nu = \tau \circ p^m.
   \end{equation}
   Thus, we have $g(x,y) \in \C[x]$ and therefore $\Prep(g) = \Prep(h)$.

   Suppose, on the other hand, that $h$ is a special polynomial. Then after a conjugation, we may assume $h(x) = S_d$, where $S_d$ is either a power map of degree $d$ or one of $\{\pm T_d\}$ for some positive integer $d > 1$. Thus, for any $y_0 \in \Per(f)$, we again have that 
   \begin{equation}
       G_{n_0-1,y_0} = \sigma_1  \circ S_{\deg(G_{n_0-1,y_0})},
   \end{equation}
   where $n_0$ is the period of $y_0$ and $\sigma_1$ is an automorphism fixing the Julia set of $S_d$ which is either an interval or a circle. Now, the exactly  same argument from above by looking at $G_{n_0-1,y_1}$, where $y_1 = f(y_0)$ would give,
   $$ g(x,y_0) = \sigma'_1 \circ S_{\deg_x(g)} \circ \nu^{-1}_{y_0},$$
   where $\nu_{y_0}$ and $\sigma'_1$ are in $\mathcal{G}(S_{\deg_x(g)})$. 
   
   If $S_{\deg_x(g)} \in \{\pm T_{\deg_x(g)}\}$, then $$\mathcal{G}(S_{\deg_x(g)}) = \{\pm x\}$$ and we have again by the pigeonhole principle that there are infinitely many $y_0$ such that 
   $$ g(x,y_0) = \tau \circ T_{\deg_x(g)} \circ \nu^{-1}$$
   where $\tau, \nu \in \{\pm x\}$. Thus, we are done in this case. 
   
   If $S_{\deg_x(g)}$ is a power map, then from the assumption that $g(x,y)$ has a constant leading coefficient as a polynomial in $\C[y][x]$, we have that for infinitely many $y_0 \in \Per(f)$, $$g(x,y_0) = cx^{\deg_x(g)} $$ for a constant $c$ doesn't depend on $y_0$. Therefore, $g(x,y) \in \C[x]$ and also $\Prep(g) = \Prep(h)$.

\end{proof}

\begin{lem}\label{lem: toward-Tits-alternative}
    Let $F_1(x,y) = (f_1(x), g_1(x,y)) $ and $F_2(x,y) = (f_2(x), g_2(x,y))$ be two polynomial skew-products  defined over $\C$ such that $\deg(f_1)=\deg_y(g_1)$, $  \deg_y(g_2) = \deg(f_2)$ and $g_1(x,y)$, $g_2(x,y)$ have constant leading coefficients as polynomials in $y$. Let $\langle F_1, F_2 \rangle$ be the semigroup generated by composition with generators $F_1$ and $F_2$. 
    
    Suppose $\Prep(F_1) = \Prep(F_2)$. Then for any pair of distinct $H_1, H_2 \in \langle F_1, F_2 \rangle$, there exists two elements $F_3(x,y)= (f(x), g_3(x,y))$ and $F_4(x,y)= (f(x),g_4(x,y))$ in the semigroup $\langle H_1, H_2 \rangle$, which can be represented by two distinct sequence of composition of $H_1$ and $H_2$, such that 
    for any $x_0 \in \Per(f)(\C)$, we have $$\Prep(G_{3,n_0-1,x_0}(y)) = \Prep(G_{4,n_0-1,x_0}(y)),$$ where $n_0$ is the period of $x_0$ and 
    $$ G_{3,n_0-1,x_0} = g_{3,n_0-1, x_0} \circ \dots \circ g_{3,0, x_0},$$
    $$ G_{4,n_0-1,x_0} = g_{4,n_0-1, x_0} \circ \dots \circ g_{4,0, x_0}.$$
\end{lem}
\begin{proof}
    Since we have $\Prep(F_1) = \Prep(F_2)$, in particular we have $\Prep(f_1) = \Prep(f_2)$. Then the first coordinates of $H_1$ and $H_2$, denoted as $h_1$ and $h_2$, also share the same set of preperiodic points. Then \cite[Corollary 4.11]{BHPT} implies that $\langle h_1, h_2 \rangle$ is not a free semigroup under composition and, hence, there exits two sequence of integers $i_1, \dots, i_{n_1}$ and $j_1, \dots, j_{n_2}$ in $\{1,2\}$, representing two distinct sequences of compositions of $h_1$ and $h_2$, such that 
    $$ f=  f_{i_1} \circ \dots \circ f_{i_{n_1}} = f_{j_1} \circ \dots \circ f_{j_{n_2}} \in \langle h_1, h_2 \rangle.$$
    Let's denote 
    $$ F_3 = F_{i_1} \circ \dots \circ F_{i_{n_1}} = (f(x), g_3(x,y)),$$
    $$ F_4 = F_{j_1} \circ \dots \circ F_{j_{n_2}} = (f(x), g_4(x,y)),$$
    where both $F_3$ and $F_4$ are in $\langle H_1, H_2 \rangle$. We may assume $F_3 \neq F_4$ as otherwise the Lemma will trivially hold.
    
    We are left to show that for every $x_0 \in \Per(f)(\C)$, we have $$\Prep(G_{3,n_0-1,x_0}(y)) = \Prep(G_{4,n_0-1,x_0}(y)),$$ where $n_0$ is the period of $x_0$. Fix a $x_0 \in \Per(f)(\C)$. Let $L$ be a finitely generated field extension of $\Q$ containing $x_0$. Let's denote $$S_1 = \Prep(F_1)(L) = \Prep(F_2)(L) \subseteq L^2.$$ Notice that $\Prep(f) = \Prep(f_1) = \Prep(f_2)$ and $S_1$ is a finite set by the Northcott property of the height function introduced by Moriwaki \cite{Mor00}. Let $S_2 \subseteq S_1$ be the subset of $S_1$ with the first coordinate equal to $x_0$. We assume that $S_2$ is not empty by enlarging $L$ if necessary. Notice that, We have $F_3(S_{1}) \subseteq S_{1}$ and $F_4(S_1) \subseteq S_1$. Now, 
    $$ F^{n_0}_3(x,y) = (f^{n_0}(x), G_{3,n_0-1,x}(y)), $$
    $$F^{n_0}_4 (x, y) = (f^{n_0}(x), G_{4,n_0-1,x}(y)).$$
    Thus 
    $$F^{n_0}_3(S_2) \subseteq S_2,$$
    $$F^{n_0}_4(S_2) \subseteq S_2,$$
    since $f^{n_0}(x_0) = x_0$. Let $\pi_y(S_2) \subseteq L$ denote the projection of $S_2$ to the second coordinate. We have 
    $$ G_{3,n_0-1,x_0}(\pi_y(S_2)) \subseteq \pi_y(S_2),$$
    $$ G_{4,n_0-1,x_0}(\pi_y(S_2)) \subseteq \pi_y(S_2).$$
    Therefore $\pi_y(S_2) \subseteq \Prep(G_{3,n_0-1,x_0}) \cap \Prep(G_{4,n_0-1,x_0})$. Notice that if we enlarge the field $L$, the size of $\pi_y(S_2)$ will keep increasing. Since the above argument holds for any large enough finitely generated field extension $L$ of $\Q$, we conclude that $G_{3, n_0-1,x_0}$ and $G_{4,n_0-1,x_0}$ share infinitely many preperiodic points. Then the unlikely intersection \cite[Theorem 1.2]{BD11} implies that 
    $$ \Prep(G_{3, n_0-1,x_0}) = \Prep(G_{4, n_0-1,x_0}).$$
\end{proof}
\begin{cor}\label{cor:auxillarycorsharedprep}
    Let $F_1(x,y) = (f_1(x), g_1(x,y))$ and $F_2(x,y) = (f_2(x), g_2(x,y))$ be polynomial skew products defined over $\C$ such that 
    $\deg(f_1)=\deg_y(g_1)$, $  \deg_y(g_2) = \deg(f_2)$ and $g_1(x,y)$ and $g_2(x,y)$ have constant leading coefficients as a polynomial in $\C[x][y]$. Then $\Prep(F_1) = \Prep(F_2)$ implies the semigroup $\langle F_1, F_2 \rangle$ generated by composition doesn't contain non-abelian free subsemigroup.
    
\end{cor}
\begin{proof}
    Borrowing the notations from Lemma \ref{lem: toward-Tits-alternative}, Lemma \ref{lem: toward-Tits-alternative} implies that for any subsemigroup of $\langle F_1, F_2 \rangle$ generated by two distinct elements, denoted as $H_1$ and $H_2$, we can find two elements $F_3$ and $F_4 \in \langle H_1, H_2 \rangle$, represented by two distinct sequences of compositions of $H_1$ and $H_2$, such that 
    $$ F_3(x,y) = (f(x), g_3(x,y)),$$
    $$ F_4(x,y) = (f(x), g_4(x,y)),$$
    and, moreover, for any $x_0 \in \Per(f)$, we have $$\Prep(G_{3,n_0-1,x_0}(y)) = \Prep(G_{4,n_0-1,x_0}(y)),$$ where $n_0$ is the period of $x_0$. Now, Proposition \ref{prop: key-prop-alternative} implies that $\langle F_3, F_4 \rangle$ is not a free semigroup under composition. Thus, $\langle H_1, H_2 \rangle$ is not free. Then we conclude that $\langle F_1, F_2 \rangle$ doesn't contain any non-abelian free subsemigroup.
\end{proof}
\begin{prop}\label{prop: orthogonal-fiber-case}
    Suppose $F_1(x,y) = (f_1(x), g_1(x,y))$ and $F_2(x,y) = (g_2(x,y), f_2(y))$ are two polynomial skew products defined over $\C$. Suppose $g_1$ and $g_2$ have constant leading coefficients as polynomials in $\C[x][y]$ and $\C[y][x]$ respectively and $\deg(f_1) = \deg_y(g_1)$, $\deg(f_2) = \deg_x(g_2)$. Suppose $\Prep(F_1) = \Prep(F_2)$, then we have the semigroup $\langle F_1, F_2 \rangle$ generated by compositions doesn't contains a non-abelian free subsemigroup.
\end{prop}
\begin{proof}
    The assumption implies that for any $y_0 \in \Per(f_2)$, we have 
    $$ G_{2, n_0-1, y_0} = g_2(x, f^{n_0-1}(y_0)) \circ \dots \circ g_2(x, y_0)$$
    sharing the same set of preperiodic points with $f_1(x)$. Thus, Lemma \ref{lem: reduce-to-same-fiberation} implies $g_2(x,y)$ is actually a function only depending on $x$, i.e., $g_2(x,y) = g_2(x)$. Therefore, $F_2(x,y) = (g_2(x), f_2(y))$, where $\deg(g_2) = \deg(f_2)$, and applying Corollary \ref{cor:auxillarycorsharedprep}, we conclude the proof of the proposition.
\end{proof}
\begin{cor}\label{cor: general-skew-product-case}
    Let $F_1$ and $F_2$ be two regular polynomial skew products defined over $\C$. If $\Prep(F_1) = \Prep(F_2)$, then the semigroup $\langle F_1, F_2 \rangle$ generated by compositions doesn't contain a non-abelian free subsemigroup.
\end{cor}
\begin{proof}
    If for some choice of coordinate $(x,y)$ of $\A^2$, $F_1(x,y) = (f_1(x), g_1(x,y))$ and $F_2(x,y) = (f_2(x), g_2(x,y))$, i.e., they are preserving the same $\A^1$ fiberation, then we can directly apply Corollary \ref{cor:auxillarycorsharedprep} to conclude the proof. Suppose otherwise, that there exists the choices of coordinates $(x,y)$ and $(z,w)$ such that $F_1(x,y) = (f_1(x),g_1(x,y))$ and $F_2(z,w) = (f_2(z), g_2(z,w))$. If $V(x)$ and $V(z)$ are parallel to each other then we are back to the first case by representing $(z, w)$ as a function of $(x,y)$. So, we assume that $(x,z)$ gives another coordinate of $\A^2$. Thus, under this new coordinate, we can write 
    $$ F_1(x,z) = (f_1(x), h_1(x,z)),$$
    $$ F_2(x,z) = (h_2(x,z), f_2(z)),$$
    where $f_1 \in \C[x]$, $f_2 \in \C[z]$, $h_1(x,z), h_2(x,z) \in \C[x,z]$ such that they have constant leading coefficients as elements in $\C[x][z]$ and $\C[z][x]$ respectively and $\deg(f_1) = \deg_z(h_1)$, $\deg(f_2) = \deg_x(h_2)$. Now, we can apply Proposition \ref{prop: orthogonal-fiber-case} to conclude the proof.
\end{proof}

\begin{rmk}\label{rmk: Tits-alternative-skew}
    Notice that for $F_1$ and $F_2$ satisfying the assumptions of Corollary \ref{cor: general-skew-product-case} above, if $\Prep(F_1) \neq \Prep(F_2)$ then there exists a positive integers $i$ such that $$\langle F^i_1, F^i_2\rangle$$ is a free subsemigroup by \cite[Theorem 1.3]{BHPT}. Therefore, the semigroup $\langle F_1, F_2 \rangle$ containing a non-abelian free subsemigroup or not is completely determined by whether $\Prep(F_1)$ and $\Prep(F_2)$ are agree.
\end{rmk}
\section{Arithmetic dynamics of polynomial skew products}\label{sec:arithmeticdynpolyskewprods}
\subsection{Arithmetic properties of canonical height for  polynomial skew products} Most of presented results (especially, the construction of height associated to polynomial skew products) in this section are largely inspired from \cite[\S 1]{DFR23} whose proof are included for the completeness for the reader in Appendix \ref{appenA}.\\

Let $(x,y)$ be affine coordinates on the affine place $\A^2_{\C}$. Recall that a polynomial skew product on $\C^2$ of degree $d\geq 2$ is a map of the form $f(x,y)=(p(x),q(x,y))$, where $p$ and $q$ are polynomials of degree $d$.
For each $n\geq 1$, we write $f^n(x,y)=(p^n(x), q^n(x,y)).$
\begin{prop}\label{prop:Nullstellensatz} For each place $v\in M_K$, there exist constants $C_v, C_v'\geq 1$ such that \begin{equation}
    C'_v\leq \frac{\max\{|p(x)|_v,|q(x,y)|_v\}}{\max\{|x|_v,|y|_v\}^d}\leq C_v \label{inequality:boundforskewproduct}
\end{equation}
for all $x,y\in \C_v$.  Moreover, for all but finitely many $v$, we can take $C_v=C'_v=1.$
\end{prop}
This readily yields the $v$-adic Green's function associated to the  polynomial skew product $f$.
\begin{cor}\label{cor:uniformconvergenceofGreen} The sequence of function 
$$g_{f,v,n}(x,y):=\frac{1}{d^n}\log^+\|f^n(x,y)\|_v$$ converges uniformly on $\C^2_v$ to a continuous function $g_{f,v}(x,y).$ It is customary to denote $\log^+(\cdot):=\max\{\log(\cdot),0\}$.
\end{cor}
We record some useful properties of the function $g_{f,v}(x,y)$.
\begin{prop} \label{prop:localgreenproperties} For each place $v\in M_K$, the continuous function $g_v: \C_v^2\rightarrow\R_{\geq 0}$ satisfies
\begin{enumerate}
\item $g_{f,v}\circ f (x,y)=dg_{f,v}(x,y)$;
\item $|g_{f,v}(x,y)-\log^+\|(x,y)\|_v|\leq C_v/(d-1)$;
\item $\{(x,y)\in \C_v^2 : g_{f,v}(x,y)=0\}=\{(x,y)\in \C_v^2: \|f^n(x,y)\|_v=O(1)\}$.
\end{enumerate}
\end{prop} To ease notation, we write $g_f$ in place of $g_{f,\infty}$ when $v$ is archimedean. 
Consider a lift $F$ on $\C^3_v$ of a polynomial skew product  $f$ that extends to $\P^2_{\C_v}$. More precisely, we write $\tilde{p}(x,z)=z^dp\left(\frac{x}{z}\right)$ and $\tilde{q}(x,y,z)=z^dq\left(\frac{x}{z},\frac{y}{z}\right).$ The degree $d$ homogeneous map $F$ which lifts $f$ to $\C^3_v$ is $$F(x,y,z)=(\tilde{p}(x,z),\tilde{q}(x,y,z),z^d).$$ In homogeneous coordinates, the inequality (\ref{inequality:boundforskewproduct}) can be expressed as \begin{equation}
    C'_v\leq \frac{\max\{|\tilde{p}(x,z)|_v, |\tilde{q}(x,y,z)|_v, |z|^d_v\}}{\max\{|x|_v,|y|_v,|z|_v\}}\leq C_v.\notag
\end{equation} Define the homogeneous local Green function $G_{F,v}:\C^3_v\backslash\{(0,0,0)\}\rightarrow\R$ associated to the global lift $F$ by   $$G_{F,v}(x,y,z)=\lim_{n\rightarrow\infty} \frac{1}{d^n}\log\|F^n(x,y,z)\|_v.$$  In addition, we obtain some useful properties of the homogeneous local dynamical height function:
\begin{prop}\label{prop:homogeneousgreenfunctionproperties} The function $G_{F,v}$ satisfies the following properties: 
\begin{enumerate} \item $G_{F,v}$ is continuous on $\C^3\backslash\{(0,0,0)\}$, $G_{F,\infty}$ is a plurisubharmonic function on $\C^3$, and the convergence is uniform on a compact subset of $\C^3_v\backslash\{(0,0,0)\}$;
\item $G_{F,v}(\lambda z,\lambda w, \lambda t)=\log|\lambda|+G_{F,v}(x,y,z)$;
\item $G_{F,v}\circ F=dG_{F,v}$;
\item $|G_{F,v}(x,y,z)-\log\max\{|x|_v,|y|_v,|z|_v\}|\leq \log C_v/(d-1)$;
\item $\mathcal{K}_{F,v}=\{(x,y,z)\in \C^3_v: G_{F,v}(x,y,z)\leq0\}$ is the filled Julia set of $F$.
\end{enumerate} 
\end{prop}
Consider a line bundle $\mathcal{O}(1)\rightarrow\P^2_K$. Here we identify the space of sections on $\mathcal{O}_{\P^2}(1)$ with the space of linear polynomials of the form $a_0x+a_1y+a_2z$ where each $a_i\in K$. On the trivialization of the line bundle over $\{x\neq 0\}$ the section is given by $\frac{a_0x+a_1y+a_2z}{x}$. Likewise for affine charts $\{y\neq 0\}$ and $\{z\neq 0\}$. We can endow a line bundle $\mathcal{O}_{\P^2}(1)$ with a metric which induced by $G_{F,v}$ in the sense that for any section $s:=a_0x+a_1y+a_2z$ $$\|s\|_v([x:y:z])=|a_0x+a_1y+a_2z|_ve^{-G_{F,v}}.$$
It is clear that the definition is well-defined because $G_{F,v}$ scales logarithmically Proposition \ref{prop:homogeneousgreenfunctionproperties}(2). That is, for $X=[x:y:z]\in \P^2_K$ and $\lambda\in K^*$
\begin{align*}\|s\|_v(\lambda X)&=|s(\lambda X)|_ve^{-G_{F,v}(\lambda X)}\\&=|\lambda|_v|s(X)|_ve^{-\log|\lambda|_v}e^{-G_{F,v}}\\&=|s(X)|_ve^{-G_{F,v}}=\|s\|_v(X).\end{align*}
For each $\mathcal{X}:=(x,y,z)\in (\overline{K})^3\backslash\{(0,0,0)\}$, we define 
$$h_{F}(\mathcal{X}):=\frac{1}{\mathrm{deg(\mathcal{X})}}\sum_{\mathcal{X}'\in \mathrm{Gal}(\overline{K}/K)\cdot \mathcal{X}}\sum_{v\in M_K}N_vG_{F,v}(\mathcal{X}').$$ As usual,  $\mathrm{Gal}(\overline{K}/K)$ is the absolute Galois group of $\overline{K}$ over $K$ and $\deg \mathcal{X}:=|\mathrm{Gal}(\overline{K}/K)\cdot \mathcal{X}|$ is the cardinality of the Galois orbit of $\mathcal{X}$. It is worthwhile to point out that $h_F$ is well-defined by the product formula. That is, $h_F(x,y,z)=h_F(\lambda x,\lambda y,\lambda z)$ for any $\lambda \in \overline{K}$. As a result, we also have a well-defined $h_f:\P^2_{\overline{K}}\rightarrow\R$. 
\begin{prop}\label{prop:propertiesofheightkewprodpoly} For any $(x,y)\in \A^2_{\overline{K}}$, we write
    $$h_f(x,y):=\frac{1}{\deg (x,y)}\sum_{(x',y')\in \mathrm{Gal}(\overline{K}/K)\cdot (x,y)}\left(\sum_{v\in M_K}N_vg_{f,v}(x',y')\right).$$ The canonical height function $h_f:\A^2_{\overline{K}}\rightarrow\R$ associated to the polynomial skew product $f$ satisfies the properties:
\begin{enumerate}
    \item $h_f(x,y)$ is non-negative.
    \item $h_f\circ f(x,y)=(\deg f) h_f(x,y)$.
    \item $h_f(x,y)-h(x,y)=O(1)$ where $h$ is the Weil height on $\A^2_{\overline{K}}$.
    \item $\mathrm{PrePer}(f)=\{h_f=0\}$.
\end{enumerate}
\end{prop}
The proof of items (1), (2), and (3) are standard while (4) follows from Northcott's theorem (cf. \cite[Theorem 3.22]{Sil07}). As one might expect, preperiodic points of polynomial skew products can be detected locally at each place $v\in M_K$. 
\begin{lem}\label{lem:loctoglobskewprodpoly} (Local-to-global principle for  polynomial skew products) Let $f$ be any  polynomial skew product defined over a number field $K$. Then $(x,y)\in \A^2_{\overline{K}}$ is $f$-preperiodic if and only if all Galois conjugates of $(x,y)$ are contained in in the filled Julia set $ \{g_{f,v}=0\}$ for all $v\in M_K$.
\end{lem}
For archimedean $v$ ($\C_v\cong \C$), the positive current $T:=dd^c g_{f}$ extends to the positive current on $\P^2_{\C}$ which we still denoted by $T$. To see this, the homogeneous Green's function $G_F$ uniquely determines the positive $(1,1)$-current $T$ on $\P^2_{\C}$ by $\pi^* T:=dd^c G_F$ where $\pi :\C^{3}\backslash\{(0,0,0)\}\rightarrow\P^2$ is the projection. We have $T=dd^c(G_F\circ s)$ where $s$ is a holomorphic section of $\pi$ on an open set. The current $T$ is independent of the holomorphic section $s$.  The support of $T$ is the complement of the Fatou set of $f$. The wedge product $\mu_f:=T \wedge T$ is a well-defined equilibrium measure on the compact set $\{g_f=0\}$.  This is the complex Monge-Amp$\grave{\text{e}}$re measure in the sense of Bedfor-Taylor \cite{BT76}.  For more details, the reader may consult \cite[\S 3, \S 8]{FS99} and \cite[\S 1]{Jo99}. 

For nonarchimedean $v$, Chambert-Loir \cite[\S 1.3]{CL11} constructed a non-archimedean analogue of the complex Monge-Amp$\grave{\text{e}}$re measure $\mu_f$. Roughly speaking, 
the (non-archimedean) Monge-Amp$\grave{\text{e}}$re measures of the function $g_{f,v}$ is assigned to be the positive measure  $c_1(\mathcal{O}_{\P^2}(1),\|\cdot\|_{g_{f,v}})^{\wedge 2}$. Note that  the metric $\|\cdot\|_{g_{f,v}}$ on line bundle $\mathcal{O}_{\P^2}(1)$ is associated to $g_{f,v}$ via $\|s\|_{g_{f,v}}:=e^{-g_{f,v}}$ where the section $s$  corresponds to the constant function $1$ on affine analytic space $\A^{2,\mathrm{an}}_{\C_v}\hookrightarrow \P^{2,\mathrm{an}}_{\C_v}$.\\

We assert that the set  $\{g_{f,v}(x,y)=0\}$ is the polynomially convex hull of $\mathrm{supp}(\mu_{f,v})$.
\begin{lem}\label{lem:largestconvexhullsupport} For each $v\in M_K$, the filled Julia set $\mathcal{K}_{\C_v}(f)$ of polynomial skew product $f$ is the largest compact set in $\A^{2,\mathrm{an}}_{\C_v}$ such that $$\sup_{\mathcal{K}_{\C_v}(f)}|P|=\sup_{\mathrm{supp}(\mu_{f,v})}|P|$$ for all $P\in \C_v[X,Y]$.
\end{lem}
\begin{prop}\label{prop:smallpointsofpolyskewprods} Let $K$ be a number field and let $f$ be a polynomial skew product  defined over $K$ of degree $\geq2$. Suppose that $\{(x_n,y_n)\}_{n\geq 1}$ is a generic sequence of algebraic points in $\A^2_{\overline{K}}$ which satisfies $$f^n(x_n,y_n)=c(x_n,y_n)$$ for all $n\in \N$ where $c: \A^2\rightarrow\A^2$ is a polynomial map. Then 
$$\lim_{n\rightarrow\infty}h_f(x_n,y_n)=0.$$
\end{prop}
\begin{proof} The proof follows  along the same lines as explained in Proposition \ref{prop:smallpoints} and Proposition \ref{prop:smallpointsproduvar}. We sketch an idea here. Using the invariant property of height associated to skew product polynomial $f$ (cf. Proposition \ref{prop:propertiesofheightkewprodpoly}(3)), we have 
\begin{equation}h_f(f^n(x_n,y_n))=(\deg(f))^nh_f(x_n,y_n).\label{eq:smallptsskewprod1}\end{equation} Now, we can fix an embedding $\A^2_{\overline{K}}\hookrightarrow \P^2_{\overline{K}}$ and so
\begin{equation}h_f(c(x_n,y_n))\leq (\deg c)h_f(x_n,y_n)+O(1)\label{eq:smallptsskewprod2}\end{equation} (see Hindry-Silverman \cite[\S B.2]{HS00}). Combining equation (\ref{eq:smallptsskewprod1}) and inequality (\ref{eq:smallptsskewprod2}), we obtain
$$((\deg f)^n-\deg c)h_f(x_n,y_n)\leq O(1).$$ The implicit constant is independent of $n$. Taking $n\rightarrow\infty$, we conclude that the sequence $\{(x_n,y_n)\}$ is $f$-dynamically small (i.e., $h_f(x_n,y_n)\rightarrow0$).
\end{proof}
\subsection{Common zeros of polynomial skew products}
\begin{prop} \label{prop:semposadelmetricforskewprodpoly}
For any polynomial skew product $f$ of degree $\geq 2$. The collection $\{g_{f,v}\}$ defines a semipositive adelic metric on $\mathcal{O}_{\P^2}(1)$ in the sense of Definition \ref{defn:semipositadelicmetricsdef}.
\end{prop}
\begin{proof}
    This follows immediately from Proposition \ref{prop:Nullstellensatz},  Proposition \ref{prop:localgreenproperties}, Proposition \ref{prop:homogeneousgreenfunctionproperties}, and Proposition \ref{prop:propertiesofheightkewprodpoly}.
\end{proof}
Proposition \ref{prop:semposadelmetricforskewprodpoly} allows us to apply arithmetic equidistribution of Yuan and Zhang. We state here a version of points of small dynamical height. The proof is identical to Theorem \ref{thm:equidistributionofsmallpoints}.  
\begin{thm}\label{thm:equidistributionpolyskewprods} (Equidsitribution for points of small height of polynomial skew products)
Let $f$ be a polynomial skew product defined over a number field $K$ such that $\deg f\geq 2$. Suppose that $\{(x_n,y_n)\}$ is a generic sequence on $\P^2_{\overline{K}}$ such that $h_f(x_n,y_n)\rightarrow0$ with respect to the arithmetic canonical height $h_f$. Then, for any place $v\in M_K$, the sequence of probability measure supported equally on the Galois orbit of $(x_n,y_n)$
	$$\frac{1}{\deg(x_n,y_n)}\sum_{z'\in \mathrm{Gal}(\overline{K}/K)\cdot (x_n,y_n)}\delta_{z'}$$ converges weakly to $\mu_{f,v}$ on the Berkovich analytic space $\P^2_{\mathrm{Berk},\C_v}$. 
\end{thm}
\begin{proof}[Proof of Theorem \ref{thm:commonzerosofpolyskewprods}] We proceed by means of contradiction. Suppose that there exists a generic sequence $\{p_i \}_{i\geq 1}$ such that for every $i$ there exist integers $m_i$ and $n_i$ so that $F^{m_i}\neq C$ and $G^{n_i}\neq C$, and \begin{equation}F^{m_i}(p_i)=G^{n_i}(p_i)=C(p_i)\label{eq:solutionsofiteratepolyskewprods}\end{equation} Our goal is to show that $F$ and $G$ are compositionally dependent. First, we notice that $m_i$ and $n_i$ must tend to $\infty$ as $i\rightarrow\infty$. Since $F$ and $G$ are polynomials skew products with respect to some coordinate choices respectively,  proposition \ref{prop:smallpointsofpolyskewprods} asserts that the generic sequence $\{p_i\}_{i \geq 1}$ must be dynamically small with respect to both $F$ and $G$. Let's first work with a coordinate of $\A^2$ such that $p_i = (x_i, y_i)$ and $F$ is a polynomial skew product with respect to this coordinate. Then, these imply
$$\lim_{i\rightarrow\infty}h_F(p_i)=0.$$
Applying the arithmetic equidistribution of small points of polynomial skew products, Theorem \ref{thm:equidistributionpolyskewprods}, we have
$$\frac{1}{\deg(x_n,y_n)}\sum_{z'\in \mathrm{Gal}(\overline{K}/K)\cdot (x_n,y_n)}\delta_{z'} = \frac{1}{\deg(p_n)} \sum_{p \in \Gal(\overline{K}/K) \cdot p_n} \delta_p $$
converges weakly to $\mu_{F,v}$ at all places $v\in M_K$ as $n$ approaches infinity. Similarly, we can also choose a coordinate of $\A^2$ such that $G$ is a polynomial skew product, denote as $(z,w)$ and notice that $(z,w) = \sigma (x,y)$ where $\sigma$ is an affine linear automorphism defined over $K$. Then we have 
\begin{align}
   & \frac{1}{\deg(z_n,w_n)}\sum_{x'\in \mathrm{Gal}(\overline{K}/K)\cdot (z_n,w_n)}\delta_{x'} = \frac{1}{\deg(\sigma(x_n,y_n))} \sum_{x' \in \mathrm{Gal}(\overline{K}/K)\cdot \sigma(x_n,y_n)} \delta_{\sigma^{-1}(x')} \nonumber \\& = \frac{1}{\deg(x_n,y_n)} \sum_{x' \in \sigma (\mathrm{Gal}(\overline{K}/K)\cdot (x_n,y_n))} \delta_{\sigma^{-1}(x')} = \frac{1}{\deg(x_n,y_n)}\sum_{z'\in \mathrm{Gal}(\overline{K}/K)\cdot (x_n,y_n)}\delta_{z'},
\end{align}
since $\deg(x_n, y_n) = \deg(\sigma(x_n, y_n))$ as $\sigma$ is a bijection defined over $K$, and this converges weakly to $\mu_{G,v} $ at all places $v \in M_K$ as $n$ approaches infinity. Thus 
$$ \mu_{F,v} = \mu_{G,v}.$$
We know, by Lemma \ref{lem:largestconvexhullsupport}, that the $v$-adic filled Julia sets of $F$ and $G$ must coincide. For each $v\in M_K$, it must hold true that  $$\{p : g_{F,v}(p)=0\}=\{p : g_{G,v}(p)=0\}.$$ Notice that $g_{F,v}$ and $g_{G,v}$ are defined with respect to the coordinate $(x,y)$ and $(z,w)$ respectively. The local-to-global property (cf. Lemma \ref{lem:loctoglobskewprodpoly}) then implies that $F$ and $G$ must share the same set of preperiodic points, i.e., $\mathrm{Prep}(F)=\mathrm{Prep}(G)$.  Corollary \ref{cor: general-skew-product-case} yields the compositional dependence of $F$ and $G$. Hence this gives rise to a contradiction. The proof completes.
\end{proof}
\appendix
\section{Proofs of arithmetic properties associated to polynomial skew products}\label{appenA}
\begin{proof}[Proof of Proposition \ref{prop:Nullstellensatz}] This follows closely Silverman \cite[Theorem 3.11]{Sil07}. For any positive number $m$, we set $$\delta_v(m)=\begin{cases}m,&\mbox{if $v$ is archimedean}\\ 1,&\mbox{if $v$ is non-archimedean}\end{cases}.$$
For any polynomial $\phi$, we define $|\phi|_v$ $\text{to be the maximum of $v$-adic coefficients of $\phi$}.$
Consider \begin{equation}
    |p(x)|_v=\left|\sum_{N=0}^da_Nx^N\right|_v\leq\delta_v(d+1)\max_{0\leq N\leq d}|a_N|_v|x|_v^d\leq C_{1,v}|x|_v^d\label{inequality:upperboundforp}
\end{equation}
where $C_{1,v}:=\delta_v(d+1)|p|_v$. Similarly, we derive
\begin{align}
|q(x,y)|_v&=\left|\sum_{i_1+i_2=d}b_{i_1i_2}x^{i_1}y^{i_2}\right|_v\notag\\&\leq \delta_v\left(\frac{1}{2}(d+2)(d+1)\right)\max_{i_1,i_2}|b_{i_1i_2}|_v\max\{|x|_v,|y|_v\}^d\notag\\&\leq C_{2,v}\max\{|x|_v,|y|_v\}^d\label{inequality:upperboundforq}
\end{align}
where $C_{2.v}:=\delta_v\left(\frac{1}{2}(d+2)(d+1)\right)|q|_v$. Thus by inequalities (\ref{inequality:upperboundforp}) and (\ref{inequality:upperboundforq}), there exists $C_v\geq 1$ so that 
$$\max\{|p(x)|_v,|q(x,y)|_v\}\leq C_v\max\{|x|_v,|y|_v\}^d.$$ This establishes an upper bound for inequality (\ref{inequality:boundforskewproduct}). On the other hand, the lower bound of (\ref{inequality:boundforskewproduct}) is a consequence of Nullstellensatz.  There exists $e\geq 0$ such that $x^e$ is in the ideal generated by $p$, i.e., $x^e\in (p)$. Thers is a polynomial $r\in \C[x]$ so that $x^e=pr$ with degree of $r$ is $e-d$. Then
\begin{align*}
|x|^e_v=|p(x)r(x)|^e_v\leq \delta_v(1+e-d)|r|_v|x|^{e-d}_v|p(x)|_v\label{inequality:Nullforp}
\end{align*}
Dividing both sides by $|x|^{e-d}_v$, we obtain
\begin{equation}
    C'_{1,v}|x|^d_v\leq |p(x)|_v\label{inequality:Nullforp}
\end{equation} where $C'_{1,v}:=\frac{1}{\delta_v(1+e-d)|r|_v}$. Likewise for $q$, there exists $e'\geq0$ such that $x^{e'}, y^{e'}\in (q).$ Then there exists $s_1,s_2\in \C[x,y]$ satisfying $x^{e'}=qs_1$ and $y^{e'}=qs_2$. Note that $\deg s_1=\deg s_2=e'-d.$ Thus
\begin{align*}
    \max\{|x|_v,|y|_v\}^{e'}&=\max\{|qs_1|_v,|qs_2|_v\}\notag\\&\leq \delta_v(1+e-d)\max\{|s_1|_v,|s_2|_v\}\max\{|x|_v,|y|_v\}^{e'-d}|q(x,y)|_v\notag
\end{align*}  
Rearranging, we have
\begin{align}
C'_{2,v} \max\{|x|_v,|y|_v\}^d&\leq |q(x,y)|_v\label{inequality:Nulllforq}
\end{align}where $C'_{2,v}=\frac{1}{\delta_v(1+e-d)\max\{|s_1|_v,|s_2|_v\}}$. Combining inequalities (\ref{inequality:Nullforp}) and (\ref{inequality:Nulllforq}),  there exists $C'_v$ satisfying
$$C_v' \max\{|x|_v,|y|_v\}^d\leq \max\{|p(x)|_v,|q(x,y)|_v\}.$$ This completes the proof of inequality (\ref{inequality:boundforskewproduct}).  The second part of the statement basically asserts that if $f$ has a good reduction at $v$, then $$\|f(x,y)\|_v=\|(x,y)\|_v^d.$$  More details about the definition of good reduction, the reader may refer \cite[Definition 10.3]{BR10}.
\end{proof}
\begin{proof}[Proof of Corollary \ref{cor:uniformconvergenceofGreen}] It follows from standard telescoping argument. For each place $v\in M_K$,  there exists (existence due to Proposition \ref{prop:Nullstellensatz}) a uniform constant $C\geq 1$ such that \begin{equation}
    |\log^+\|f(x,y)\|_v-d\log^+\|(x,y)\|_v|\leq C\label{eq:uniformbounddifferencelogplus}
\end{equation} for all $x,y\in \C_v$. Note that we may take $C:=\max\{\log C_v,-\log C'_v\}$. Here $\log^+\|(x,y)\|_v=\log\max\{1,|x|_v,|y|_v\}.$ Now, let $n>m\geq 0$ and compute
\begin{align*}
&\left|\frac{1}{d^n}\log^+\|f^n(x,y)\|_v-\frac{1}{d^m}\log^+\|f^m(x,y)\|_v\right|\\&=\left|\sum_{i=m}^{n}\left(\frac{1}{d^i}\log^+\|f^i(x,y)\|-\frac{1}{d^{i-1}}\log^+\|f^{i-1}(x,y)\|_v\right)\right|\\&\leq \sum_{i=m}^n\frac{1}{d^i}|\log^+\|f^i(x,y)\|_v-d\log^+\|f^{i-1}(x,y)\|_v|\\&\leq \sum_{i=m}^n\frac{C}{d^i}\quad (\text{using}\,\,(\ref{eq:uniformbounddifferencelogplus}))
\end{align*}
The last quantity can be made arbitrary small. Thus the sequence $\{d^{-n}\log^+\|f^n(x,y)\|_v\}$ is Cauchy in $\R$ and hence converges which we denote the limiting as 
$$g_v(x,y)=\lim_{n\rightarrow\infty}\frac{1}{d^n}\log\max\{1,|p(x)|_v,|q(x,y)|_v\}.$$ It is clear that $g_v(x,y)$ is continuous on $\C^2_v$ as the sequence $g_{v,n}$ is continuous. Now, we see that (1) and (3) are clear from the limit definition of $g_{f,v}$. Notice that (3) asserts that the local Green's function of polynomial skew product $f$ vanish exactly the set at which $f$ has bounded forward orbit.  The statement (2) is basically taking $m=0$ in the telescoping argument. 
\end{proof}
\begin{proof}[Proof of Proposition \ref{prop:homogeneousgreenfunctionproperties}] For (1), it can be derived similarly as in \cite[Theorem 2.1]{HP94}. Roughly speaking, the uniform convergence is a standard telescoping argument. For archimedean $v=\infty$, we know that a limit of plurisubharmonic function is again plurisubharmonic, so $G_{F,v}$ is plurisubharmonic. Also, the contiuity of $G_{F,v}$ follows as being the uniform limit of continuous funciton. For (2), we can easily compute
\begin{align*}
G_{F,v}(\lambda x,\lambda y,\lambda z)&=\lim_{n\rightarrow\infty} \frac{1}{d^n}\log \max\{|(\lambda t)^{d^n}p^n(x/z)|_v,(\lambda z)^{d^n}q^n(x/z,y/z),|(\lambda z)^{d^n}|_v\}\\&=\lim_{n\rightarrow\infty}\frac{1}{d^n}\log|\lambda^{d^n}|_v+\lim_{n\rightarrow\infty}\log\max\{|z^{d^n}p^n(x/z)|_v,|z^{d^n}q(x/z,y/z)|_v|z^{d^n}|_v\}
\\&=|\lambda|_v+G_{F,v}(x,y,z).\end{align*} Now (3) is an easy algebra, that is 
\begin{align*}G_{F,v}(F(x,y,z))&=\lim_{n\rightarrow\infty}\frac{1}{d^n}\log\|F^n(F(x,y,z))\|_v\\&=\lim_{n\rightarrow\infty}\frac{d}{d^{n+1}}\log\|F^{n+1}(x,y,z)\|_v=dG_{F,v}(x,y,z).\end{align*}
The inequality (4) is a basic telescoping argument. For (5), the filled Julia set $K_{F,v}$ can be interpreted as the unit ball in $\C^3_v$ with respect to the homogeneous Green function $G_{F,v}$ in light of \cite[Lemma 3.7]{BR06}. The proof of our statement follows  their idea along the same line. In detail, if $(x,y,z)\in K_{F,v}$, there exists $M>0$ such that $\|F^n(x,y,z)\|_v\leq M$ for all $n$ and therefore $G_{F,v}(x,y,z)\leq \displaystyle\lim_{n\rightarrow\infty}\frac{1}{d^n}\log M=0$. For the reverse inequality, we may suppose $(x,y,z)\not\in K_{F,v}$. Then for sufficiently large enough $n_0$ there is a radius $R_v$ such that $\|F^{n_0}(x,y,z)\|_v>R_v$. An easy induction argument yields the desired inequality.
\end{proof}
\begin{proof}[Proof of Lemma \ref{lem:loctoglobskewprodpoly}] For any $(x,y)\in \A^2_{\overline{K}}$, we take a finite extension $L$ of $K$ so that  $(x,y)\in \A^2_L$. Then proposition \ref{prop:propertiesofheightkewprodpoly}(4) asserts that

   $$ (x,y)\in \mathrm{Prep}(f) \Longleftrightarrow h_f(x,y)=0.$$
Recall that the canonical height function $h_f$ is decomposed as the local Green's function $g_{f,v}$ and $g_{f,v}$ is a nonnegative function. The fact that $h_f(x,y)=0$ implies $g_{f,v}(x^{\sigma},y^{\sigma})=0$ for all $v\in M_K$ and all $\sigma : L^2\hookrightarrow \C^2_v.$  Now, Proposition \ref{prop:localgreenproperties}(3) ensures that all Galois conjugates of $(x,y)$ are contained the $v$-adic filled Julia set of $f$ at all places $v\in M_K$.  
\end{proof}
\begin{proof}[Proof of Lemma \ref{lem:largestconvexhullsupport}] The proof idea is exactly the same as described in \cite[Lemma 6.3]{DF17}. We include here for the convenience of the reader. First note that $\mathrm{supp}(\mu_{f,v})\subset \mathcal{K}_{\C_v}(f)$ and so $\displaystyle\sup_{\mathrm{supp}(\mu_{f,v})}|P|\leq \displaystyle\sup_{\mathcal{K}_{\C_v}(f)}|P|.$ It remains to verify the reverse inequality. Set $$C_0=\sup_{\mathrm{supp}(\mu_{f,v})}|P|.$$ Let $\varepsilon>0.$ We then define $$\widetilde{g_{f,v}}:=\max\left\{g_{f,v},\log\frac{|P|}{C_0+\varepsilon}\right\}.$$ The function $\widetilde{g_{f,v}}$ is continuous non-negative function on $\A^2_{\C_v}$ which induces a continuous semipositive metric on $\mathcal{O}(1).$ Near $\mathrm{supp}(\mu_{f,v})$, we have $\widetilde{g_{f,v}}=g_{f,v}$. Applying \cite[Corollay A.2]{DF17}, the Monge-Amp$\grave{\text{e}}$re measures 
 of $g_{f,v}$ and $\widetilde{g_{f,v}}$ must coincide. Hence $\widetilde{g_{f,v}}-g_{f,v}$ must be constant by a result of Yuan-Zhang \cite[Corollary 2.2 (Calabi Theorem)]{YZ17}, so $\widetilde{g_{f,v}}=g_{f,v}$. It follows that $$\log\frac{|P|}{C_0+\varepsilon}\leq g_{f,v}.$$ On the filled Julia set $\mathcal{K}_{\C_v}(f)$, we obtain 
$$\log\frac{|P|}{C_0+\varepsilon}\leq 0.$$ It is not hard to see that $$\sup_{\mathcal{K}_{\C_v}(f)}|P|\leq \sup_{\mathrm{supp}(\mu_{f,v})}|P|+\varepsilon.$$ Letting $\varepsilon\rightarrow0$, we conclude $$\sup_{\mathcal{K}_{\C_v}(f)}|P|\leq \sup_{\mathrm{supp}(\mu_{f,v})}|P|$$ as desired.
\end{proof}

\end{document}